\documentclass[12pt,reqno]{amsart}
\usepackage{amssymb,amsfonts,amsthm}
\usepackage{ulem}
\usepackage{amsthm}
\usepackage{amssymb}
\usepackage[english,american]{babel}
\usepackage{enumerate}
\usepackage{setspace}
\usepackage{hyperref}
\usepackage{color}

\definecolor{Red}{cmyk}{0,1,1,0.2}

\usepackage{amsfonts,amsmath,layout}

\newcommand{\R}{\mathbb R}

\def\R{\mathbb R}

\newcommand{\be}{\begin{equation}}
\newcommand{\ee}{\end{equation}}

\def\1{{\bf 1}}

\newtheorem{Theorem}{Theorem}[section]
\newtheorem{Definition}[Theorem]{Definition}
\newtheorem{Proposition}[Theorem]{Proposition}
\newtheorem{Lemma}[Theorem]{Lemma}

\begin{document}
\title[Quasi linear parabolic pde posed on a network]{Quasi linear parabolic pde posed on a network with non
linear Neumann boundary condition at vertices}
\author{Isaac Ohavi}
\email{isaacohavi@gmail.com}
\dedicatory{Version: \today}
\maketitle
\begin{abstract} 
The purpose of this article is to study quasi linear parabolic partial differential equations of second order, posed on a bounded network, satisfying a nonlinear and non dynamical Neumann boundary condition at the vertices. We prove the existence and the uniqueness of a classical solution.
\end{abstract}
\section{Introduction}\label{sec intro}
The purpose of this article is to study quasi linear parabolic partial differential equations of second order, posed on a bounded network, satisfying a nonlinear
and non dynamical Neumann boundary condition at the vertices, also called the junction points.
For simplicity of notation, we have focused on a network containing a single junction.
More precisely, giving a final time $T>0$, if $I\in \mathbb{N}^*$ denotes the number of edges of the single junction, $a_i>0$ is the length of each edge $i$ $(i\in \{1\ldots I\}$), and $"0"$ denotes the junction point, the problem is reduced to:
\begin{eqnarray}\label{eq de base}
\begin{cases}\partial_tu_i(t,x)-\sigma_i(x,\partial_xu_i(t,x))\partial_x^2u_i(t,x)+\\
H_i(x,u_i(t,x),\partial_xu_i(t,x))~~=~~0,~~\text{ if } (t,x) \in (0,T)\times (0,a_i),\\
F(u(t,0),\partial_x u(t,0))~~=~~0, ~~ \text{ if } t\in[0,T),\\
\text{with}~~u(t,0)=(u_1(t,0),\ldots,u_I(t,0)),~~\partial_x u(t,0)=(\partial_xu_1(t,0),\ldots,\partial_xu_I(t,0)),\\
\text{and}~~\forall (i,j)\in\{1\ldots I\}^2,~~u_i(t,0)=u_j(t,0),\\
\forall i\in\{1\ldots I\},~~ u_i(t,a_i)~~=~~\phi_i(t),~~ \text{ if } t \in [0,T],\\
\forall i\in\{1\ldots I\},~~ u_i(0,x)~~=~~g_i(x),~~ \text{ if } x\in [0,a_i].\\
\end{cases}
\end{eqnarray}
To simplify our study, we impose Dirichlet boundary condition ($\phi_i$) at the $a_i$. One can consider classical Neumann boundary conditions, without adding any mathematical technical issues for the problem \eqref{eq de base}.
The multi-junction setting involving the equations of type \eqref{eq de base} can be treated with similar tools, with a Neumann boundary condition $F$ (or a classical Dirichlet boundary condition) at the junction points.
The well-known Kirchhoff law corresponds to the case where $F$ is linear in $\partial_xu$ and independent of $u$, (for instance: $F(\partial_x u(t,0))=\sum_{i=1}^I \partial_xu_i(t,0)=0$).

Originally  introduced by Nikol'skii \cite{anal nikolski} and Lumer \cite{Lumer 1, Lumer 2}, the concept of ramified spaces and the analysis of partial differential equation on these spaces have attracted a lot of attention in the last 30 years. As explained in  \cite{anal nikolski}, the main motivations are applications in physics, chemistry, and biology (for instance small transverse vibrations in a grid of strings, vibration of a grid of beams, drainage system, electrical equation with Kirchhoff law, wave equation, heat equation,...). 
Concerning applications in biology, we can cite for instance the recent works (\cite{DLPZ},\cite{JPS} and \cite{Vasilyeva}), where equations of type \eqref{eq de base} are used in the multi-junction case to modelize the dynamic of species in river networks.
Linear diffusions of the form \eqref{eq de base}, with a Kirchhoff law, are also naturally associated with stochastic processes living on graphs. These processes were introduced in the seminal papers \cite{freidlinS} and \cite{Freidlin}. Another motivation for studying \eqref{eq de base} is the analysis of associated stochastic optimal control problems with a control at the junction.  

There have been several works on linear and quasilinear parabolic equations of the form \eqref{eq de base}. For linear equations, von Below \cite{Below 0}  shows that, under natural smoothness and compatibility conditions, linear boundary value problems posed on a junction with a linear Kirchhoff condition at the junction point is well-posed. The proof consists mainly in showing that the initial boundary value problem posed on a junction is equivalent to a well-posed initial boundary value problem for a parabolic system, where the boundary conditions are such that the classical results on linear parabolic equations \cite{pde para} can be applied. The same author investigates in \cite{Below 3} the strong maximum principle for semi linear parabolic operators with Kirchhoff condition, while in \cite{Below 4} he studies  the classical global solvability for a class of semilinear parabolic equations on ramified networks, where a dynamical node condition is prescribed: Namely the Neumann condition at the junction point $x=0$ in \eqref{eq de base}, is replaced by the dynamic one:
\begin{eqnarray*}
\partial_tu(t,0) + F(t,u(t,0),\partial_xu(t,0))~~=~~0.
\end{eqnarray*}
In this way the application of classical estimates for domains established in \cite{pde para} becomes possible. The author then establishes the classical solvability in the class $\mathcal{C}^{1+\alpha,2+\alpha}$, with the aid of the Leray-Schauder-principle and the maximum principle of \cite{Below 3}. Let us  note that this kind of proof fails for equation \eqref{eq de base} because in this case one cannot expect an uniform bound for the term $|\partial_tu(t,0)|$ (the proof of Lemma 3.1 of \cite{pde para} VI.3 fails). Still in the linear setting, another approach, yielding similar existence results, was developed  by Fijavz, Mugnolo and Sikolya in \cite{M-K}: the idea is to use semi-group theory as well as variational methods to understand how the spectrum of the operator is related to the  structure of the network. 
 
Equations of the form \eqref{eq de base} can also be analyzed in terms of viscosity solutions. The first results on viscosity solutions for Hamilton-Jacobi equations on networks have been obtained by Schieborn in \cite{D-S these} for the Eikonal equations and later discussed in many contributions on first order problems \cite{Camilli 1, Imbert Nguyen, Lions College France}, elliptic equations \cite{Lions Souganidis 1} and second order problems with vanishing diffusion at the vertex \cite{Lions Souganidis 2}. 
In contrast second order Hamilton-Jacobi equations with a non vanishing viscosity at the boundary have seldom been studied in the literature and our aim is to show the well-posedness of classical solutions for \eqref{eq de base} in suitable H\"{o}der spaces: see Theorem \ref{th : exis para} for the existence and  Theorem \ref{th : para comparison th} for the comparison, and thus the uniqueness. Our main assumptions are that the equation is uniformly parabolic with smooth coefficients and that the term $F=F(u,p)$ at the junction is either decreasing with respect to $u$ or increasing with respect to $p$.
The main idea of the proof is to use a time discretization, exploiting at each step the solvability in $C^{2+\alpha}$ of the elliptic problem:
\begin{eqnarray}\label{eq base 2}
\begin{cases}
-\sigma_i(x,\partial_x u_i(x))\partial_x^2u_i(x)+H_i(x,u_i(x),\partial_xu_i(x))~~=~~0,~~\text{if}~~x\in(0,a_i) \\
F(u(0),\partial_xu(0))~~=~~0,\\
\text{with}~~u(0)=(u_1(0),\ldots,u_I(0)),~~\partial_x u(0)=(\partial_xu_1(0),\ldots,\partial_xu_I(0)),\\
\text{and}~~\forall (i,j)\in\{1\ldots I\}^2,~~u_i(0)=u_j(0),\\
\forall i\in\{1 \ldots I\},~~u_i(a_i)=\phi_i.
\end{cases}
\end{eqnarray}

The paper is organized as follows. In section \ref{sec main results}, we introduce the notations and state our main results. In Section \ref{sec review elliptic problem}, we review the mains results of existence and uniqueness of the elliptic problem \eqref{eq base 2}. Finally Section \ref{sec para probl}, is dedicated to the proof of our main results.
\section{main results}\label{sec main results}
In this section we state our main result Theorem \ref{th : exis para}, on the solvability of the parabolic problem with Neumann boundary condition at the vertex, posed on a bounded junction \eqref{eq de base}, stated in Introduction:
\begin{eqnarray*}\begin{cases}\partial_tu_i(t,x)-\sigma_i(x,\partial_xu_i(t,x))\partial_x^2u_i(t,x)+\\
H_i(x,u_i(t,x),\partial_xu_i(t,x))~~=~~0,~~\text{ if } (t,x) \in (0,T)\times (0,a_i),\\
F(u(t,0),\partial_x u(t,0))~~=~~0, ~~ \text{ if } t\in[0,T),\\
\text{with}~~u(t,0)=(u_1(t,0),\ldots,u_I(t,0)),~~\partial_x u(t,0)=(\partial_xu_1(t,0),\ldots,\partial_xu_I(t,0)),\\
\text{and}~~\forall (i,j)\in\{1\ldots I\}^2,~~u_i(t,0)=u_j(t,0),\\
\forall i\in\{1\ldots I\},~~ u_i(t,a_i)~~=~~\phi_i(t),~~ \text{ if } t \in [0,T],\\
\forall i\in\{1\ldots I\},~~ u_i(0,x)~~=~~g_i(x),~~ \text{ if } x\in [0,a_i].\\
\end{cases} \eqref{eq de base}
\end{eqnarray*}
There will be two typical assumptions for $F= F(u,p)$: either $F$ is decreasing with respect to $u$ or $F$ is increasing with respect to $p$ (Kirchhoff conditions).

\subsection{Notations and preliminary results}\label{subec main result}

Let us start by introducing the main notation used in this paper as well as an interpolation result.
Let $I\in \mathbb{N}^*$ be the number of edges, and $a=(a_1,\ldots a_I)\in (0,\infty)^I$ be the length of each edge. The bounded junction $\mathcal{J}^a$ is defined by
\begin{eqnarray*}\mathcal{J}^a &  = & \bigcup_{i =1}^I J_{i}^{a_i},\text{ with: } \forall i\in \{1\ldots I\}~~J_{i}^{a_i}:=[0,a_i],~~\text{and}~~\forall (i,j)\in \{1\ldots I\}^2,~~i\neq j,~~J_{i}^{a_i}\cap J_{j}^{a_j}=\{0\}.
\end{eqnarray*}
The intersection of the $(J_{i}^{a_i})_{1 \leq i\leq I}$ is called the junction point and is denoted by $0$.\\
We identify  all the points of $\mathcal{J}^a$ by the couples $(x,i)$ (with $i \in\{1\ldots I\} , x\in|0,\max_{i\in\{1\ldots I\}} a_i]$), such that we have: $(x,i)\in \mathcal{J}^a$ if and only if $x\in J_{i}^{a_i}$.
For $T>0$, the time-space domain $\mathcal{J}^a_T$ is defined by
\begin{eqnarray*}\mathcal{J}^a_T &  = & [0,T]\times\mathcal{J}^a.
\end{eqnarray*}
The interior of $\mathcal{J}_{T}^a$ set minus the junction point $0$ is denoted by $\overset{\circ}{\mathcal{J}_{T}^a}$, and is defined by
\begin{eqnarray*}
\overset{\circ}{\mathcal{J}_{T}^a}~~=~~(0,T)\times\Big(\bigcup_{i=1}^I\overset{\circ}{J_{i}^{a_i}} \Big). 
\end{eqnarray*}
For the functionnal spaces that will be used in the sequel, we  use here the notations of Chapter 1.1 of \cite{pde para}. For the convenience of the reader, we recall these notations in Appendix \ref{sec : functionnal spaces}.
In addition we introduce the  parabolic H\"{o}lder space on the junction $\Big(\mathcal{C}^{\frac{l}{2},l}(\mathcal{J}^a_T),\|.\|_{\mathcal{C}^{\frac{l}{2},l}(\mathcal{J}^a_T)}\Big)$ and the space $\mathcal{C}_b^{\frac{l}{2},l}(\overset{\circ}{\mathcal{J}^a_T})$, defined by (where $l>0$, see Annexe \ref{sec : functionnal spaces} for more details)
\begin{eqnarray*}
&\mathcal{C}^{\frac{l}{2},l}(\mathcal{J}^a_T)~~:=\Big\{~~f:\mathcal{J}^a_T\to\R,~~(t,(x,i))\mapsto f_i(t,x), ~~ 
\forall (i,j)\in \{1\ldots I\}^{2},~~ \forall t\in(0,T),\\&~ f_i(t,0)=f_j(t,0),~~\forall i\in \{1\ldots I\},~~
(t,x)\mapsto f_i(t,x)\in \mathcal{C}^{\frac{l}{2},l}([0,T]\times[0,a_i])~~\Big\},\\
&\mathcal{C}_b^{\frac{l}{2},l}(\overset{\circ}{\mathcal{J}^a_T})~~:=\Big\{~~f:\mathcal{J}^a_T\to\R,~~(t,(x,i))\mapsto f_i(t,x),\\&~~\forall i\in \{1\ldots I\},~~
(t,x)\mapsto f_i(t,x)\in \mathcal{C}_b^{\frac{l}{2},l}((0,T)\times(0,a_i))~~\Big\},
\end{eqnarray*}
with:
\begin{eqnarray*}
\|u\|_{\mathcal{C}^{\frac{l}{2},l}(\mathcal{J}^a_T)}~~=~~\sum_{1\leq i\leq I} \|u_i\|_{\mathcal{C}^{\frac{l}{2},l}([0,T]\times[0,a_i])}.
\end{eqnarray*}
We will use the same notations, when the domain does not depend on time, namely $T=0$, $\Omega_T=\Omega$, removing the dependence on the time variable.

We continue with the definition of a nondecreasing map $F:\R^I\to \R$.
Let $(x=(x_1,\ldots x_I),y=(y_1\ldots y_I))\in \R^{2I}$, we say that 
\begin{eqnarray*}
x \leq y, \text{ if }~~ \forall i\in \{1\ldots I\},~~ x_i\leq y_i,
\end{eqnarray*}
and
\begin{eqnarray*}
x<y, \text{ if }~~x\leq y,\text{ and there exists } j\in \{1\ldots I\}, ~~x_j<y_j .
\end{eqnarray*}
We say that $F\in \mathcal{C}(\R^I,\R)$ is nondecreasing if
\begin{eqnarray*}
\forall (x,y)\in \R^I, \text{ if } x\leq y, \text{ then } F(x)\leq F(y),
\end{eqnarray*}
increasing if
\begin{eqnarray*}
\forall (x,y)\in \R^I, \text{ if } x<y, \text{ then } F(x)<F(y).
\end{eqnarray*}Next we recall an interpolation inequality, which will be useful in the sequel.
\begin{Lemma}\label{lm : cont deruiv temps}
Suppose that $u \in \mathcal{C}^{0,1}([0,T]\times[0,R])$ satisfies an H\"{o}lder condition in $t$ in $[0,T]\times [0,R]$, with exponent $\alpha\in (0,1]$, constant $\nu_1$, and has derivative $\partial_xu$, which for any $t\in[0,T]$ are H\"{o}lder continuous in the variable $x$, with exponent $\gamma\in(0,1]$, and constant $\nu_2$. Then the derivative $\partial_xu$ satisfies in $[0,T]\times[0,R]$, an H\"{o}lder condition in $t$, with exponent $\frac{\alpha\gamma}{1+\gamma}$, and constant depending only on  $\nu_1,\nu_2,\gamma$. More precisely 
\begin{eqnarray*}
&\forall (t,s)\in[0,T]^2,~~|t-s|\leq 1,~~\forall x\in[0,R],\\
&|\partial_xu(t,x)-\partial_xu(s,x)|~~~\leq ~~\Big(2\nu_2\Big(\frac{\nu_1}{\gamma \nu_2}\Big)^{\frac{\gamma}{1+\gamma}}~~+~~2\nu_1\Big(\frac{\gamma\nu_2}{ \nu_1}\Big)^{-\frac{1}{1+\gamma}}\Big)|t-s|^{\frac{\alpha\gamma}{1+\gamma}}.
\end{eqnarray*}
\end{Lemma}
This is a special case of Lemma II.3.1, in \cite{pde para}, (see also \cite{anal nikolski}). The main difference is that we are able to get global H\"{o}lder regularity in $[0,T]\times[0,R]$ for $\partial_xu$ in its first variable. Let us recall that this kind of result fails in higher dimensions.
\begin{proof}
Let $(t,s)\in [0,T]^2$, with $|t-s|\leq 1$, and $x\in[0,R]$.
Suppose first that $x\in [0,\frac{R}{2}]$.
Let $y\in[0,R]$, with $y\neq x$, we write:
\begin{eqnarray*}
&\partial_xu(t,x)-\partial_xu(s,x)~~=\\
&~~\displaystyle\frac{1}{y-x}\int_x^y~~(\partial_xu(t,x)-\partial_xu(t,z))+(\partial_xu(t,z)-\partial_xu(s,z))+(\partial_xu(s,z)-\partial_xu(s,x))~~dz.
\end{eqnarray*}
Using the H\"{o}lder condition in time satisfied by $u$, we have:
\begin{eqnarray*}
\Big|\frac{1}{y-x}\int_x^y~~(\partial_xu(t,z)-\partial_xu(s,z))dz\Big|~~\leq~~\frac{2\nu_1|t-s|^\alpha}{|y-x|}.
\end{eqnarray*}
On the other hand, using the H\"{o}lder regularity of $\partial_xu$ in space satisfied, we have: 
\begin{eqnarray*}
\Big|\frac{1}{y-x}\int_x^y~~(\partial_xu(t,x)-\partial_xu(t,z))+(\partial_xu(s,z)-\partial_xu(s,x))dz\Big|~~\leq~~2\nu_2|y-x|^\gamma.
\end{eqnarray*}
It follows:
\begin{eqnarray*}
|\partial_xu(t,x)-\partial_xu(s,x)|~~\leq~~2\nu_2|y-x|^\gamma~~+~~\frac{2\nu_1|t-s|^\alpha}{|y-x|}.
\end{eqnarray*}
Assuming that $|t-s|\leq \Big((\frac{3R}{2})^{1+\gamma}\frac{\gamma\nu_2}{\nu_1}\Big)^{\frac{1}{\alpha}}\wedge 1$, minimizing in $y\in[0,R]$, for $y>x$, the right side of the last equation, we get that the infimum is reached for 
\begin{eqnarray*}
y^*~~=~~x~~+~~\Big(\frac{\nu_1|t-s|^\alpha}{\gamma \nu_2}\Big)^{\frac{1}{1+\gamma}},
\end{eqnarray*}
and then:
\begin{eqnarray*}
|\partial_xu(t,x)-\partial_xu(s,x)|~~\leq ~~C(\nu_1,\nu_2,\gamma)|t-s|^{\frac{\alpha\gamma}{1+\gamma}},
\end{eqnarray*}
where the constant $C(\nu_1,\nu_2,\gamma)$, depends only on the data $(\nu_1,\nu_2,\gamma)$,  and is given by:
\begin{eqnarray*}
C(\nu_1,\nu_2,\gamma)~~=~~2\nu_2\Big(\frac{\nu_1}{\gamma \nu_2}\Big)^{\frac{\gamma}{1+\gamma}}~~+~~2\nu_1\Big(\frac{\gamma\nu_2}{ \nu_1}\Big)^{-\frac{1}{1+\gamma}}.
\end{eqnarray*}
For the cases $y<x$, and $x\in [\frac{R}{2},R]$, we argue similarly, which completes the proof.
\end{proof}
\subsection{Assumptions and main results}
We state in this subsection the central Theorem of this note, namely the solvability and uniqueness of $\eqref{eq de base}$ in the class $\mathcal{C}^{\frac{\alpha}{2},1+\alpha}(\mathcal{J}^a_T)\cap \mathcal{C}_b^{1+\frac{\alpha}{2},2+\alpha}(\overset{\circ}{\mathcal{J}^{a}_T})$.\textbf{ In the rest of these notes, we fix $\alpha \in(0,1)$}.

We introduce the following data 
$$\begin{cases}
F\in \mathcal{C}^0(\R^I\times\R^I,\R)\\
g\in\mathcal{C}^{1}(\mathcal{J}^a)\cap\mathcal{C}_b^{2}(\overset{\circ}{\mathcal{J}^{a}})
\end{cases},$$
and for each $i\in\{1\ldots I\}$
$$\begin{cases}
\sigma_i \in \mathcal{C}^{1}([0,a_i]\times \R,\R)\\
H_i \in \mathcal{C}^1([0,a_i]\times \R^2,\R)\\ 
\phi_i \in \mathcal{C}^1([0,T],\R)
\end{cases}.$$
We suppose furthermore that the data satisfy the following assumption
$$\textbf{Assumption } (\mathcal{P})$$
(i) Assumption on $F$
$$\begin{cases}
a)~~F \text{ is decreasing with respect to its first variable,} \\
b)~~F \text{ is nondecreasing with respect to its second variable,}\\
c)~~\exists(b,B)\in\R^I\times\R^I,~~F(b,B)=0,
\end{cases}$$
or satisfies the Kirchhoff condition
$$\begin{cases}
a)~~F \text{ is nonincreasing with respect to its first variable},\\
b)~~F \text{ is increasing with respect to its second variable,}\\
c)~~\exists(b,B)\in\R^I\times\R^I,~~F(b,B)=0.
\end{cases}$$
We suppose moreover that there exists a parameter $m\in \R$, $m\ge 2$ such that we have\\
(ii) The (uniform) ellipticity condition on the $(\sigma_i)_{i\in\{1\ldots I\}}$ : there exist $\underline{\nu},\overline{\nu}$, strictly positive constants such that:
\begin{eqnarray*}
\forall i\in \{1\ldots I\},~~\forall (x,p)\in [0,a_i]\times\R,\\
\underline{\nu}(1+|p|)^{m-2}~~\leq~~\sigma_i(x,p)~~\leq~~\overline{\nu}(1+|p|)^{m-2}.
\end{eqnarray*}
(iii) The growth of the $(H_i)_{i\in\{1\ldots I\}}$ with respect to $p$ exceed the growth of the $\sigma_i$ with respect to $p$ by no more than two, namely there exists $\mu$ an increasing real continuous function such that:
\begin{eqnarray*}
\forall i\in \{1\ldots I\},~~\forall (x,u,p)\in [0,a_i]\times\R^2,~~|H_i(x,u,p)|~~\leq~\mu(|u|)(1+|p|)^{m}.
\end{eqnarray*}
(iv) We impose the following restrictions on the growth with respect to $p$ of the derivatives for the coefficients $(\sigma_i,H_i)_{i\in\{1\ldots I\}}$, which are for all $i\in \{1\ldots I\}$:
\begin{eqnarray*}
&a)~~|\partial_p\sigma_i|_{[0,a_i]\times\R^2}(1+|p|)^2+|\partial_pH_i|_{[0,a_i]\times\R^2}~~\leq~~\gamma(|u|)(1+|p|)^{m-1},\\
&b)~~|\partial_x\sigma_i|_{[0,a_i]\times\R^2}(1+|p|)^2+|\partial_xH_i|_{[0,a_i]\times\R^2}~~\leq~~\Big(\varepsilon(|u|)+P(|u|,|p|)\Big)(1+|p|)^{m+1},\\
&c)~~\forall (x,u,p)\in [0,a_i]\times\R^3,~~-C_H~~\leq~~\partial_uH_i(x,u,p)~~\leq~~\Big(\varepsilon(|u|)+P(|u|,|p|)\Big)(1+|p|)^{m},
\end{eqnarray*}
where $\gamma$ and $\varepsilon$ are continuous non negative increasing functions. $P$ is a continuous function, increasing with respect to its first variable, and tends to $0$ for $p\to +\infty$, uniformly with respect to its first variable, from $[0,u_1]$ with $u_1\in R$, and $C_H>0$ is real strictly positive number. We assume that $(\gamma,\varepsilon,P,C_H)$ are independent of $i\in \{1\ldots I\}$.\\
(v) A compatibility conditions for $g$ and $(\phi_i)_{\{1\ldots I\}}$:
\begin{eqnarray*}
&F(g(0),\partial_xg(0))~~=~~0~~;~~\forall i\in \{1\ldots I\},~~g_{i}(a_i)=\phi_i(0).
\end{eqnarray*}
\begin{Theorem}\label{th : exis para}
Assume $({\mathcal P})$. Then system \eqref{eq de base} is uniquely solvable in the class $\mathcal{C}^{\frac{\alpha}{2},1+\alpha}(\mathcal{J}^a_T)\cap \mathcal{C}_b^{1+\frac{\alpha}{2},2+\alpha}(\overset{\circ}{\mathcal{J}^{a}_T})$.
There exist constants $(M_1,M_2,M_3)$, depending only the data introduced in assumption $(\mathcal{P})$,
\begin{eqnarray*}
&M_1=M_1\Big(\max_{i\in\{1\ldots I\}}\Big\{~~\sup_{x\in(0,a_i)}|-\sigma_i(x,\partial_xg_i(x))\partial_x^2g_i(x)+H_{i}(x,g_i(x),\partial_xg_i(x))|+\\
&|\partial_t\phi_i|_{(0,T)}~~\Big\},\max_{i\in\{1\ldots I\}} |g_i|_{(0,a_i)},C_H\Big),\\
&M_2=M_2\Big(\overline{\nu},\underline{\nu},\mu(M_1),\gamma(M_1),\varepsilon(M_1),\sup_{|p|\ge 0}P(M_1,|p|),\max_{i\in\{1\ldots I\}}|\partial_xg_i|_{(0,a_i)},M_1\Big),
\\
&M_3=M_3\Big(M_1,\underline{\nu}(1+|p|)^{m-2},\mu(|u|)(1+|p|)^m,~~|u|\leq M_1,~~|p|\leq M_2\Big),\\
\end{eqnarray*}
such that
\begin{eqnarray*}
&||u||_{\mathcal{C}(\mathcal{J}_T^a)}\leq M_1,~~||\partial_xu||_{\mathcal{C}(\mathcal{J}_T^a)}\leq M_2,~~||\partial_tu||_{\mathcal{C}(\mathcal{J}_T^a)}\leq M_1,~~||\partial^2_{x}u||_{\mathcal{C}(\mathcal{J}_T^a)}\leq M_3.
\end{eqnarray*}
Moreover, there exists a constant $M(\alpha)$ depending on $\Big(\alpha,M_1,M_2,M_3\Big)$ such that 
\begin{eqnarray*}
&||u||_{\mathcal{C}^{\frac{\alpha}{2},1+\alpha}(\mathcal{J}_T^a)}\leq M(\alpha).
\end{eqnarray*}
\end{Theorem}
We continue this Section by giving the definitions of super and sub solution, and stating a comparison Theorem for our problem. 
\begin{Definition}\label{def : sur/sous solutions}
We say that $u\in \mathcal{C}^{0,1}(\mathcal{J}^a_T)\cap \mathcal{C}^{1,2}(\overset{\circ}{\mathcal{J}^a_T})$, is a super solution (resp. sub solution) of
\begin{eqnarray}\label{eq : Neumann par }
\begin{cases} \partial_tu_i(t,x)-\sigma_i(x,\partial_xu_i(t,x))\partial_x^2u_i(t,x)+\\
H_i(x,u_i(t,x),\partial_xu_i(t,x))~~=~~0,~~ \text{ if } (t,x) \in (0,T)\times (0,a_i),\\
F(u(t,0),\partial_x u(t,0))~~=~~0, ~~ \text{ if } t\in(0,T),
\end{cases}
\end{eqnarray}
if 
\begin{eqnarray*}
\begin{cases}
\partial_tu_i(t,x)-\sigma_i(x,\partial_xu_i(t,x))\partial_x^2u_i(t,x)+\\
H_i(x,u_i(t,x),\partial_xu_i(t,x))~~\ge ~~0,~~(\text{resp.} \leq 0),~~\forall (t,x) \in (0,T)\times(0,a_i),\\
F(u(t,0),\partial_x u(t,0))~~\leq~~0,~~(\text{resp.} \ge 0),~~\forall t\in (0,T)\\
\end{cases}
\end{eqnarray*}
\end{Definition}
\begin{Theorem}\label{th : para comparison th} Parabolic comparison.
~\\
Assume $(\mathcal{P})$. Let $u\in \mathcal{C}^{0,1}(\mathcal{J}^a_T)\cap \mathcal{C}^{1,2}_b(\overset{\circ}{\mathcal{J}^a_T})$ (resp. $v\in \mathcal{C}^{0,1}(\mathcal{J}^a_T)\cap \mathcal{C}^{1,2}_b(\overset{\circ}{\mathcal{J}^a_T})$) a super solution (resp. a sub solution) of \eqref{eq : Neumann par }, satisfying for all $i\in\{1\ldots I\}$, $u_i(t,a_i)\ge v_i(t,a_i)$, for all $t\in [0,T]$, and $u_i(0,x)\ge v_i(0,x)$, for all $x\in [0,a_i]$.\\
Then for each $(t,(x,i))\in\mathcal{J}^a_T$ : $u_i(t,x)\ge v_i(t,x)$.
\end{Theorem}
\begin{proof}
We start by showing that for each $0\leq s<T$, for all $(t,(x,i))\in \mathcal{J}^a_s$, $u_i(t,x)\ge v_i(t,x)$.\\
Let $\lambda>0$. Suppose that $\lambda > C_1+C_2$, where the expression of the constants $(C_1,C_2)$ are given in the sequel (see \eqref{eq : 1cont comp}, and \eqref{eq : 2const compa}).
We argue by contradiction assuming that
\begin{eqnarray*}
\sup_{(t,(x,i))\in \mathcal{J}^a_{s}} \exp(-\lambda t+x)\Big(v_{i}(t,x)-u_{i}(t,x)\Big)>0.
\end{eqnarray*}
Using the boundary conditions satisfied by $u$ and $v$, the supremum above is reached at a point $(t_0,(x_0,j_0))\in (0,s]\times \mathcal{J}$, with $0\leq x_0<a_{j_0}$.

Suppose first that $x_0>0$, the optimality conditions imply that 
\begin{eqnarray*}
&\exp(-\lambda t_0+x_0)\Big(-\lambda(v_{j_0}(t_0,x_0)-u_{j_0}(t_0,x_0))~~+~~\partial_tv_{j_0}(t_0,x_0)-\partial_tu_{j_0}(t_0,x_0)\Big)~~\ge~~0,\\
&\exp(-\lambda t_0+x_0))\Big(v_{j_0}(t_0,x_0)-u_{j_0}(t_0,x_0)~~+~~\partial_xv_{j_0}(t_0,x_0)-\partial_xu_{j_0}(t_0,x_0)\Big)~~=~~0,\\
&
\exp(-\lambda t_0+x_0)\Big(v_{j_0}(t_0,x_0)-u_{j_0}(t_0,x_0)~~+~~2\Big(\partial_xv_{j_0}(t_0,x_0)-\partial_xu_{j_0}(t_0,x_0)\Big)\\
&+~~\Big(\partial_x^2v_{j_0}(t_0,x_0)-\partial_x^2u_{j_0}(t_0,x_0)\Big)\Big)~~=\\
&\exp(-\lambda t_0+x_0)\Big(-\Big(v_{j_0}(t_0,x_0)-u_{j_0}(t_0,x_0)\Big)~~~+~~\partial_x^2v_{j_0}(t_0,x_0)-\partial_x^2u_{j_0}(t_0,x_0)\Big)~~\leq~~0.
\end{eqnarray*}
Using assumptions $(\mathcal{P})$ (iv) a), (iv) c) and the optimality conditions above we have
\begin{eqnarray*}
&H_{j_0}(x_0,u_i(t_0,x_0),\partial_xu_{j_0}(t_0,x_0))-H_{j_0}(x_0,v_{j_0}(t_0,x_0),\partial_xv_{j_0}(t_0,x_0))~~\leq\\
&\Big(v_{j_0}(t_0,x_0)-u_{j_0}(t_0,x_0)\Big)\Big(C_H+\gamma(|\partial_xv_{j_0}(t_0,x_0)|)\Big)\Big((1+|\partial_xu_{j_0}(t_0,x_0))|\vee|\partial_xv_{j_0}(t_0,x_0))|)^{m-1}\Big)\\
&\leq~~C_1\Big(v_{j_0}(t_0,x_0)-u_{j_0}(t_0,x_0)\Big),
\end{eqnarray*}
where 
\begin{eqnarray}\label{eq : 1cont comp}
\nonumber  &C_1~~:=~~
\max_{i\in\{1\ldots I\}}~~\Big\{~~\sup_{(t,x)\in[0,T]\times[0,a_i]}\Big\{~~\Big(C_H+
\gamma(|\partial_xv_{i}(t,x)|\Big)\Big(1+|\partial_xu_{i}(t,x))|\\
&\vee|\partial_xv_{i}(t,x))|\Big)^{m-1}~~\Big\}~~\Big\}.
\end{eqnarray}
On the other hand we have using assumption $(\mathcal{P})$ (ii), (iv) a), (iv) c), and the optimality conditions
\begin{eqnarray*}
&\sigma_{j_0}(x_0,\partial_xv_{j_0}(t_0,x_0))\partial_x^2v_{j_0}(t_0,x_0)-\sigma_{j_0}(x_0,\partial_xu_{j_0}(t_0,x_0))\partial_x^2u_{j_0}(t_0,x_0)~~\leq\\
&\Big(v_{j_0}(t_0,x_0)-u_{j_0}(t_0,x_0)\Big)
\Big(~~\overline{\nu}(1+|\partial_xv_{j_0}(t_0,x_0)|)^{m-2}~~+~~\Big|\partial_x^2u_{j_0}(t_0,x_0)\Big|\\
&+~~\gamma(|\partial_xu_{j_0}(t_0,x_0)|)(1+|\partial_xu_{j_0}(t_0,x_0))|\vee|\partial_xv_{j_0}(t_0,x_0))|)^{m-1}~~\Big)\\
&\leq~~C_2\Big(v_{j_0}(t_0,x_0)-u_{j_0}(t_0,x_0)\Big),
\end{eqnarray*}
where 
\begin{eqnarray}\label{eq : 2const compa}
&\nonumber C_2~~:=~~\max_{i\in\{1\ldots I\}}~~\Big\{~~\sup_{(t,x)\in[0,T]\times[0,a_i]}\Big\{~~\overline{\nu}(1+|\partial_xv_{i}(t,x)|)^{m-2}+\Big|\partial_x^2u_{i}(t,x)\Big|\\
&+~~\gamma(|\partial_xu_{i}(t,x)|)(1+|\partial_xu_{i}(t,x))|+|\partial_xv_{i}(t,x))|)^{m-1}~~\Big\}~~\Big\}.
\end{eqnarray}
Using now the fact that $v$ is a sub-solution while $u$ is a super-solution, we get
\begin{eqnarray*}
& 0 \leq \\
& \partial_tu_{j_0}(t_0,x_0) -\sigma_{j_0}(x_0,\partial_xu_{j_0}(t_0,x_0))\partial_x^2u_{j_0}(t_0,x_0)+H_{j_0}(x_0,u_i(t_0,x_0),\partial_xu_{j_0}(t_0,x_0))\\
&-\partial_tv_{j_0}(t_0,x_0) +\sigma_{j_0}(x_0,\partial_xv_{j_0}(t_0,x_0))\partial_x^2v_{j_0}(t_0,x_0)-H_{j_0}(x_0,v_{j_0}(t_0,x_0),\partial_xv_{j_0}(t_0,x_0))\\
&\leq -(\lambda-(C_1+C_2))(v_{j_0}(t_0,x_0)-u_{j_0}(t_0,x_0))~~<~~0,
\end{eqnarray*}
which is a contradiction.
Therefore the supremum is reached at $(t_0,0)$, with $t_0\in(0,s]$.
We apply a first order Taylor expansion in space, in the neighborhood of the junction point $0$. Since for all $(i,j)\in\{1\ldots I\}$, $u_i(t_0,0)=u_j(t_0,0)$, and $v_i(t_0,0)=v_j(t_0,0)$, we get from
\begin{eqnarray*}
&\forall (i,j)\in\{1,\ldots I\}^2,~~\forall h\in(0,\min_{i\in\{1\ldots I\}}a_i]\\
&v_j(t_0,0)-u_j(t_0,0)~~\ge~~\exp(h)\Big(v_i(t_0,h)-u_i(t_0,h)\Big),
\end{eqnarray*}
that
\begin{eqnarray*}
&\forall (i,j)\in\{1,\ldots I\}^2,~~\forall h\in(0,\min_{i\in\{1\ldots I\}}a_i] \\
&v_j(t_0,0)-u_j(t_0,0)~~\ge~~v_i(t_0,0)-u_i(t_0,0)~~+~~\\
&h\Big(~~v_i(t_0,0)-u_i(t_0,0)~~+~~\partial_xv_i(t_0,0)-\partial_xu_i(t_0,0)~~\Big)~~+~~h\varepsilon_i(h),
\end{eqnarray*}
where 
\begin{eqnarray*}
&\forall i\in\{1,\ldots I\},~~\lim_{h\to 0}\varepsilon_i(h)=0.
\end{eqnarray*}
We get then
\begin{eqnarray*}
\forall i\in\{1,\ldots I\},~~\partial_xv_i(t_0,0) ~~\leq ~~\partial_xu_i(t_0,0)-\Big(v_i(t_0,0)-u_i(t_0,0)\Big)~~<~~\partial_xu_i(t_0,0).
\end{eqnarray*}
Using the growth  assumptions on $F$ (assumption $(\mathcal{P})$(i)), and the fact that $v$ is a sub-solution while $u$ is a super-solution, we get
\begin{eqnarray*}
0~~\leq~~F(v(t_0,0),\partial_x v(t_0,0)) ~~ < ~~F(u(t_0,0),\partial_x u(t_0,0))~~\leq~ 0,
\end{eqnarray*}
and then a contradiction.

We deduce then for all $0\leq s<T$, for all $(t,(x,i))\in[0,s]\times\mathcal{J}^a$, 
\begin{eqnarray*}
\exp(-\lambda t+x)\Big(v_{i}(t,x)-u_{i}(t,x)\Big)~~\leq~~0.
\end{eqnarray*}
Using the continuity of $u$ and $v$, we deduce finally that for all $(t,(x,i))\in[0,T]\times\mathcal{J}^a$,
\begin{eqnarray*}
v_{i}(t,x)~~\leq~~u_{i}(t,x).
\end{eqnarray*}
\end{proof}
\section{The elliptic problem}\label{sec review elliptic problem}
As explained in the introduction, the construction of a solution for our parabolic problem \eqref{eq de base} relies on a time discretization and on the solvability of the associated elliptic problem. We review in this section the well-posedness of the elliptic problem \eqref{eq base 2}:
\begin{eqnarray*}
\begin{cases}
-\sigma_i(x,\partial_x u_i(x))\partial_x^2u_i(x)+H_i(x,u_i(x),\partial_xu_i(x))~~=~~0,~~\text{if}~~x\in(0,a_i) \\
F(u(0),\partial_xu(0))~~=~~0,~~\text{with}~~\forall (i,j)\in\{1\ldots I\}^2,~~u_i(0)=u_j(0),\\
\text{and}~~u(0)=(u_1(0),\ldots,u_I(0)),~~\partial_x u(0)=(\partial_xu_1(0),\ldots,\partial_xu_I(0)),\\
\forall i\in\{1 \ldots I\},~~u_i(a_i)=\phi_i.
\end{cases}\eqref{eq base 2}
\end{eqnarray*}
We introduce the following data for $i\in \{1\ldots I\}$ 
$$\begin{cases}
F\in \mathcal{C}^0(\R^I\times \R^I,\R)\\
\sigma_i \in \mathcal{C}^{1}([0,a_i]\times \R,\R)\\
H_i \in \mathcal{C}^{1}([0,a_i]\times\R^2,\R)\\
\phi_i \in \R
\end{cases},$$
satisfying the following assumption
$$\textbf{Assumption }(\mathcal{E})$$ 
(i) Assumption on $F$
$$\begin{cases}
a)~~F \text{ is decreasing with respect to its first variable,}\\
b)~~F\text{ is nondecreasing with respect to its second variable},\\
c)~~\exists (b,B)\in\R^I\times \R^I,  \text{ such that : }F(b,B)~~=~~0,\\
\end{cases}$$
or $F$ satisfy the Kirchhoff condition
$$\begin{cases}
a)~~F \text{ is nonincreasing with respect to its first variable,}\\
b)~~F\text{ is increasing with respect to its second variable},\\
c)~~\exists (b,B)\in\R^I\times \R^I,  \text{ such that : }F(b,B)~~=~~0.\\
\end{cases}$$
(ii) The ellipticity condition on the $\sigma_i$
$$\exists c>0,~~\forall i\in \{1\ldots I\},~~\forall (x,p)\in [0,a_i]\times \R,~~\sigma_i(x,p)\ge c.$$
(iii) For the Hamiltonians $H_i$, we suppose
\begin{eqnarray*}
\exists C_H>0,~~\forall i\in \{1\ldots I\},~~\forall (x,u,v,p)\in (0,a_i)\times\R^3,&\\~~\text{ if } u\leq v,~~C_H(u-v)\leq H_i(x,u,p)-H_i(x,v,p).&
\end{eqnarray*}
For each $i\in \{1\ldots I\}$, we define the following differential operators $(\delta_i,\overline{\delta}_i)_{i\in \{1\ldots I\}}$ acting on $\mathcal{C}^{1}([0,a_i]\times \R^2,\R)$, for $f=f(x,u,p)$ by
\begin{eqnarray*}
\delta_i~~:=~~\partial_u+\frac{1}{p}\partial_x;~~\overline{\delta_i}~~:=~~p\partial_p.
\end{eqnarray*}
(iv) We impose the following restrictions on the growth with respect to $p$ for the coefficients $(\sigma_i,H_i)_{i\in\{1\ldots I\}}=(\sigma_i(x,p),H_i(x,u,p))_{i\in\{1\ldots I\}}$, which are for all $i\in \{1\ldots I\}$
\begin{eqnarray*}
\delta_i\sigma_i~~=~~o(\sigma_i),\\
\overline{\delta_i}\sigma_i~~=~~\mathcal{O}(\sigma_i),\\
H_i~~=~~\mathcal{O}(\sigma_ip^2),\\
\delta_i H_i~~\leq~~o(\sigma_ip^2),\\
\overline{\delta_i}H_i~~\leq~~\mathcal{O}(\sigma_ip^2),
\end{eqnarray*}
where the limits behind are understood as $p\to +\infty$, uniformly in $x$, for bounded $u$.

The main result of this section is the following Theorem, for the solvability and uniqueness of the elliptic problem posed on the junction, with non linear Neumann condition at the junction point.
\begin{Theorem}\label{th:exi elli}
Assume $(\mathcal{E})$. Then system \eqref{eq base 2}, is uniquely solvable in the class $\mathcal{C}^{2+\alpha}(\mathcal{J}^a)$.
\end{Theorem}
Theorem \ref{th:exi elli} is stated without proof in \cite{Lions Souganidis 1}. For the convenience of the reader, we sketch its proof in the Appendix.
The uniqueness of the solution of \eqref{eq base 2}, is a consequence of the elliptic comparison Theorem for smooth solutions, for the Neumann problem, stated in this Section, and whose proof uses the same arguments of the proof of the parabolic comparison Theorem \ref{th : para comparison th}.

We complete this section by recalling the definition of super and sub solution for the elliptic problem \eqref{eq base 2}, and the corresponding elliptic comparison Theorem.
\begin{Definition}
Let $u\in \mathcal{C}^{2}(\mathcal{J}^a)$. We say that $u$ is a super solution (resp. sub solution) of 
\begin{eqnarray}\label{eq : Neum1}\begin{cases}  -\sigma_i(x,\partial_x f_i(x))\partial_x^2f_i(x)+H_i(x,f_i(x),\partial_x f_i(x))~~=~~0, \text{ if } x \in (0,a_i),\\
F(f(0),\partial_x f(0))~~=~~0,\\
\end{cases}\end{eqnarray}
if 
\begin{eqnarray*}
\begin{cases}
-\sigma_i(x,\partial_x u_i(x))\partial_x^2u_i(x)+H_i(x,u_i(x),\partial_x u_i(x))~~\ge ~~0,~~(\text{resp.} \leq 0),~~\text{ if } x \in (0,a_i),\\
F(u(0),\partial_x u(0))~~\leq~~0,~~(\text{resp.} \ge 0).
\end{cases}
\end{eqnarray*}
\end{Definition}
\begin{Theorem}\label{th: Ellip Th Comp} Elliptic comparison Theorem, see for instance Theorem 2.1 of \cite{Lions Souganidis 1}.
~\\
Assume ($\mathcal{E}$). Let $u\in \mathcal{C}^{2}(\mathcal{J}^a)$ (resp. $v\in \mathcal{C}^{2}(\mathcal{J}^a)$) a super solution (resp. a sub solution) of \eqref{eq : Neum1}, satisfying for all $i\in\{1\ldots I\}$, $u_i(a_i)\ge v_i(a_i)$.
Then for each $(x,i)\in\mathcal{J}^a$ : $u_i(x)\ge v_i(x)$.
\end{Theorem}
\section{The parabolic problem}\label{sec para probl}
In this Section, we prove Theorem \ref{th : exis para}.
The construction of the solution is based on the results obtained in Section \ref{sec review elliptic problem} for the elliptic problem, and is done by considering a sequence $u^n\in \mathcal{C}^{2}(\mathcal{J}^a)$, solving on a time grid an elliptic scheme defined by induction. We will prove that the solution $u^n$  converges to the required solution.
\subsection{Estimates on the discretized scheme}\label{sou sec estimation schema}
Let $n\in \mathbb{N}^*$, we consider the following time grid, $(t_k^n=\frac{kT}{n})_{0\leq k\leq n}$ of $[0,T]$,
and the following sequence $(u_{k})_{0\leq k \leq n}$ of $\mathcal{C}^{2+\alpha}(\mathcal{J}^a)$, defined recursively by
\begin{eqnarray*}
\text{for}~~k=0,~~ 
u_0=g,
\end{eqnarray*}
and for $1\leq k \leq n$, $u_{k}$ is the unique solution of the following elliptic problem 
\begin{eqnarray}\label{eq: schema ell}
\begin{cases}
n(u_{i,k}(x)-u_{i,k-1}(x))-\sigma_i(x,\partial_xu_{i,k}(x))\partial_x^2u_{i,k}(x))+\\
H_{i}(x,u_{i,k}(x),\partial_xu_{i,k}(x))~~=~~0,~~ \text{ if } x \in (0,a_i),\\ 
F(u_{k}(0),\partial_x u_k(0))~~=~~0,\\
\forall i\in\{1\ldots I\},~~ u_{i,k}(a_i)~~=~~\phi_i(t_k^n).
\end{cases}
\end{eqnarray}
The solvability of the elliptic scheme \eqref{eq: schema ell} can be proved by induction, using the same arguments as for Theorem \ref{th:exi elli}. The next step consists in obtaining uniform estimates of $(u_k)_{0\leq k\leq n}$. We start first by getting uniform bounds for $n|u_{i,k}-u_{i,k-1}|_{(0,a_i)}$ using the comparison Theorem \ref{th: Ellip Th Comp}. 
\begin{Lemma} \label{lm: esti derivée temps} 
Assume $(\mathcal{P})$. There exists a constant $C>0$, independent of $n$, depending only the data $C=C\Big(\max_{i\in\{1\ldots I\}}\Big\{~~\sup_{x\in(0,a_i)}|-\sigma_i(x,\partial_xg_i(x))\partial_x^2g_i(x)+H_{i}(x,g_i(x),\partial_xg_i(x))|+|\partial_t\phi_i|_{(0,T)}~~\Big\},C_H\Big)$, such that: 
\begin{eqnarray*}
\sup_{n\ge 0}~~\max_{k\in\{1\ldots n\}}~~\max_{i\in\{1\ldots I\}}\Big\{~~n|u_{i,k}-u_{i,k-1}|_{(0,a_i)}~~\Big\}~~\leq~~C,\\
\end{eqnarray*}
and then
\begin{eqnarray*}
\sup_{n\ge 0}~~\max_{k\in\{0\ldots n\}}~~\max_{i\in\{1\ldots I\}}\Big\{~~|u_{i,k}|_{(0,a_i)}~~\Big\}~~\leq~~C+\max_{i\in\{1\ldots I\}}\Big\{~~|g_{i}|_{(0,a_i)}~~\Big\}.
\end{eqnarray*}
\end{Lemma}
\begin{proof}
Let $n>\lfloor C_H \rfloor$, where $C_H$ is defined in assumption $(\mathcal{P})$ (iv) c). Let $k\in\{1\ldots n\}$, we define the following sequence:
\begin{eqnarray*}
\begin{cases}
M_{0}~~=~~\max_{i\in\{1\ldots I\}}\Big\{~~\sup_{x\in(0,a_i)}|-\sigma_i(x,\partial_xg_i(x))\partial_x^2g_i(x)+H_{i}(x,g_i(x),\partial_xg_i(x))|+|\partial_t\phi_i|_{(0,T)}~~\Big\},\\
M_{k,n}~~=~~\displaystyle\frac{n}{n-C_H}M_{k-1,n},~~k\in\{1\ldots n\}.
\end{cases}
\end{eqnarray*}
We claim that for each $k\in\{1\ldots n\}$:
$$\max_{i\in\{1\ldots I\}}\Big\{~~n|u_{i,k}-u_{i,k-1}|_{(0,a_i)}\Big\}~~\leq~~M_{k,n}.$$
We give a proof by induction. For this, if $k=1$, let us show that the map $h$ defined on the junction by:
\begin{eqnarray*}
h:=\begin{cases}
\mathcal{J}^a\to \R \\
(x,i)\mapsto \frac{M_{1,n}}{n}+g_i(x),
\end{cases}
\end{eqnarray*}
is a super solution of \eqref{eq: schema ell}, for $k=1$. For this we will use the Elliptic Comparison Theorem \ref{th: Ellip Th Comp}.

Using the compatibility conditions satisfied by $g$, namely assumption $(\mathcal{P})$ (v), and the assumptions of growth on $F$, assumption $(\mathcal{P})$ (i), we get for the boundary conditions: 
\begin{eqnarray*}
F(h(0),\partial_x h(0))~~\leq~~ F(g(0),\partial_x g(0))~~=~~0,&\\ ~~\forall i\in\{1\ldots I\},~~ h_i(a_i)~~=~~\frac{M_{1,n}}{n}+g_i(a_i)~~\ge~~\frac{M_{0,n}}{n}+g_i(a_i)~~\ge~~\phi_i(t_1^n).&
\end{eqnarray*}
For all $i\in\{1\ldots I\}$, and $x\in(0,a_i)$, we get using assumption $(\mathcal{P})$ (iii):
\begin{eqnarray*}
n(h_{i}(x)-g_i(x))-\sigma_i(x,\partial_xh_i(x))\partial_x^2h_i(x)+H_{i}(x,h_i(x),\partial_xh_i(x))~~=~~&\\M_{1,n}-\sigma_i(x,\partial_xg_i(x))\partial_x^2g_i(x)+H_{i}(x,\frac{M_{1,n}}{n}+g_i(x),\partial_xg_i(x))~~\ge~~ &\\
M_{1,n}-\sigma_i(x,\partial_xg_i(x))\partial_x^2g_i(x)+H_{i}(x,g_i(x),\partial_xg_i(x)) -\frac{M_{1,n}C_H}{n}~~\ge~~0.&
\end{eqnarray*} 
It follows from the comparison Theorem \ref{th: Ellip Th Comp}, that for all $i\in\{1\ldots I\}$, and $x\in[0,a_i]$:
\begin{eqnarray*}
u_{1,i}(x)~~\leq~~\frac{M_{1,n}}{n}+g_i(x).
\end{eqnarray*}
Using the same arguments, we show that: 
\begin{eqnarray*}
h:=\begin{cases}
\mathcal{J}^a\to \R \\
(x,i)\mapsto -\frac{M_{1,n}}{n}+g_i(x),
\end{cases}
\end{eqnarray*}
is a sub solution of \eqref{eq: schema ell} for $k=1$, and we then get:
\begin{eqnarray*}
\max_{i\in\{1\ldots I\}}\Big\{~~\sup_{x\in(0,a_i)}n|u_{1,i}(x)-g_i(x)|~~\Big\}~~\leq~~M_{1,n}.
\end{eqnarray*}

Let $2\leq k\leq n$, suppose that the assumption of induction holds true.Let us show that the following map: 
\begin{eqnarray*}
h:=\begin{cases}
\mathcal{J}^a\to \R \\
(x,i)\mapsto \frac{M_{k,n}}{n}+u_{i,k-1}(x),
\end{cases}
\end{eqnarray*}
is a super solution of \eqref{eq: schema ell}.
For the boundary conditions, using assumption $(\mathcal{P})$ (i), we get:
\begin{eqnarray*}
F(h(0),\partial_x h(0))~~\leq~~F(u_{k-1}(0),\partial_x u_{k-1}(0))~~\leq~~0,\\
\forall i\in\{1\ldots I\},~~h_i(a_i)~~=~~\frac{M_{k,n}}{n}+u_{i,k-1}(a_i)~~\ge~~\frac{M_{0,n}}{n}+u_{i,k-1}(a_i)~~\ge~~\phi_i(t_k^n).&
\end{eqnarray*}
For all $i\in\{1\ldots I\}$, and $x\in(0,a_i)$:
\begin{eqnarray*}
n(h_i(x)-u_{i,k-1}(x))-\sigma_i(x,\partial_xh(x))\partial_x^2h(x)+H_{i}(x,h(x),\partial_xh(x))~~=~~&\\
M_{k,n}-\sigma_i(x,\partial_{x}u_{i,k-1}(x))\partial_x^2u_{i,k-1}(x)+H_{i}(x,\frac{M_{k,n}}{n}+u_{i,k-1}(x),\partial_xu_{i,k-1}(x))~~\ge~~&\\
M_{k,n}-\sigma_i(x,\partial_{x}u_{i,k-1}(x))\partial_x^2u_{i,k-1}(x)+H_{i}(x,u_{i,k-1}(x),\partial_xu_{i,k-1}(x))-\frac{C_HM_{k,n}}{n}.&
\end{eqnarray*}
Since we have for all $x\in(0,a_i)$:
\begin{eqnarray*}
-\sigma_i(x,\partial_{x}u_{i,k-1}(x))\partial_x^2u_{i,k-1}(x)+H_{i}(x,u_{i,k-1}(x),\partial_xu_{i,k-1}(x))~~=~~-n(u_{i,k-1}(x)-u_{i,k-2}(x)), 
\end{eqnarray*}
using the induction assumption we get:
\begin{eqnarray*}
n(h_i(x)-u_{i,k-1}(x))-\sigma_i(x,\partial_xh(x))\partial_x^2h(x)+H_{i}(x,\partial_xh(x),\partial_xh(x))~~\ge~~&\\
M_{k,n}-n(u_{i,k-1}(x)-u_{i,k-2}(x))-\frac{C_HM_{k,n}}{n}~~\ge~~M_{k,n}\frac{n-C_H}{n}-M_{k-1,n}~~\ge~~0.&
\end{eqnarray*}
It follows from the comparison Theorem \ref{th: Ellip Th Comp}, that for all $(x,i)\in \mathcal{J}^a$:
\begin{eqnarray*}
u_{i,k}(x)~~\leq~~\frac{M_{k,n}}{n}+u_{i,k-1}(x).
\end{eqnarray*}
Using the same arguments, we show that: 
\begin{eqnarray*}
h:=\begin{cases}
\mathcal{J}^a\to \R \\
(x,i)\mapsto -\frac{M_{k,n}}{n}+u_{i,k-1}(x),
\end{cases}
\end{eqnarray*}
is a sub solution of \eqref{eq: schema ell}, and we get:
\begin{eqnarray*}
\max_{i\in\{1\ldots I\}}\Big\{~~n|u_{i,k}(x)-u_{i,k-1}(x)|_{(0,a_i)}~~\Big\}~~\leq~~ M_{k,n}.
\end{eqnarray*}
We obtain finally using that for all $k\in\{1\ldots n\}$:
\begin{eqnarray*}
\begin{cases}
M_{k,n}~~\leq~~M_{n,n},\\
M_{k,n}~~=~~\Big(\displaystyle\frac{n}{n-C_H}\Big)^kM_{0},
\end{cases}
\end{eqnarray*}
and 
\begin{eqnarray*}
&M_{n,n}\xrightarrow[]{n\to +\infty}~~C~~:=~~\exp(C_H)\max_{i\in\{1\ldots I\}}\Big\{~~\sup_{x\in(0,a_i)}|-\sigma_i(x,\partial_xg_i(x))\partial_x^2g_i(x)+\\&
H_{i}(x,g_i(x),\partial_xg_i(x))|+|\partial_t\phi_i|_{(0,T)}~~\Big\},
\end{eqnarray*}
that
\begin{eqnarray*}
&\sup_{n\ge 0}~~\max_{k\in\{1\ldots n\}}~~\max_{i\in\{1\ldots I\}}\Big\{~~n|u_{i,k}-u_{i,k-1}|_{(0,a_i)}~~\Big\}~~\leq~~C,\\
&\sup_{n\ge 0}~~\max_{k\in\{0\ldots n\}}~~\max_{i\in\{1\ldots I\}}\Big\{~~|u_{i,k}|_{(0,a_i)}~~\Big\}~~\leq~~C+\max_{i\in\{1\ldots I\}}\Big\{~~|g_{i}|_{(0,a_i)}~~\Big\}.
\end{eqnarray*}
That completes the proof.
\end{proof}
The next step consists in obtaining uniform estimates for $|\partial_xu_{i,k}|_{(0,a_i)}$, in terms of $n|u_{i,k}-u_{i,k-1}|_{(0,a_i)}$ and the quantities $(\underline{\nu},\overline{\nu},\mu,\gamma,\varepsilon,P)$ introduced in assumption $(\mathcal{P})$  (ii), (iii) and (iv). More precisely, we use similar arguments as for the proof of Theorem 14.1 of \cite{Gilbarg}, using a classical argument of upper and lower barrier functions at the boundary. The assumption of growth ($\mathcal{P}$) (ii) and (iii) are used in a key way to get an uniform bound on the gradient at the boundary. Finally to conclude, we  appeal to a gradient maximum principle, using the growth assumption ($\mathcal{P}$) (iv), adapting Theorem 15.2 of \cite{Gilbarg} to our elliptic scheme.
\begin{Lemma}\label{lm : borne de grad u}
Assume ($\mathcal{P}$). There exists a constant $C>0$, independent of $n$, depending only the data:
\begin{eqnarray*}
&\Big(\overline{\nu},\underline{\nu},\mu(|u|),\gamma(|u|),\varepsilon(|u|),\sup_{|p|\ge 0}P(|u|,|p|),\max_{i\in\{1\ldots I\}}|\partial_xg_i|_{(0,a_i)},\\
&\sup_{n\ge 0}\max_{k\in\{1\ldots n\}}\max_{i\in\{1\ldots I\}}\Big\{~~n|u_{i,k}-u_{i,k-1}|_{(0,a_i)}~~\Big\},\\
& \text{with}~~ |u|\leq\sup_{n\ge 0}\max_{k\in\{0\ldots n\}}\max_{i\in\{1\ldots I\}}~~\Big\{~~|u_{i,k}|_{(0,a_i)}~~\Big\}\Big),
\end{eqnarray*} such that:
\begin{eqnarray*}
\sup_{n\ge 0}~~\max_{k\in\{0\ldots n\}}~~\max_{i\in\{1\ldots I\}}~~\Big\{~~|\partial_xu_{i,k}|_{(0,a_i)}~~\Big\}~~\leq~~C.
\end{eqnarray*}
\end{Lemma}  
\begin{proof}
\textbf{Step 1:} We claim that, for each $k\in\{1\ldots n\}$, $\max_{i\in\{1\ldots I\}}\Big\{~~|\partial_xu_{i,k}|_{\partial(0,a_i)}~~\Big\}$ is bounded by the data, uniformly in $n$.

It follows from Lemma \ref{lm: esti derivée temps}, that there exists $M>0$ such that: 
\begin{eqnarray*}
\sup_{n\ge 0}~~\max_{k\in\{0\ldots n\}}~~\max_{i\in\{1\ldots I\}}~~\Big\{~~|u_{i,k}|_{(0,a_i)}+n|u_{i,k}-u_{i,k-1}|_{(0,a_i)}~~\Big\}~~\leq~~M.
\end{eqnarray*}
We fix $i\in\{1\ldots I\}$. We apply a barrier method consisting in building two functions $w^+_{i,k},w_{i,k}^-$ satisfying in a neighborhood of $0$, for example $[0,\kappa]$, with $\kappa\leq a_i$:
\begin{eqnarray*}
Q_{i}(x,w^+_{i,k}(x),\partial_xw^+_{i,k}(x),\partial_x^2w^+_{i,k}(x))\ge 0,~~\forall x\in [0,\kappa], ~~w^+_{i,k}(0)=u_{i,k}(0), ~~w^+_{i,k}(\kappa)\ge M,&\\
Q_{i}(x,w_{i,k}^-(x),\partial_xw_{i,k}^-(x),\partial_x^2w_{i,k}^-(x))\leq 0,~~\forall x\in [0,\kappa],~~ w_{i,k}^-(0)=u_{i,k}(0),~~ w_{i,k}^-(\kappa)\leq -M,&
\end{eqnarray*}
where we recall that for each $(x,u,p,S)\in [0,a_i]\times R^3$:
\begin{eqnarray*}
Q_{i}(x,u,p,S)~~=~~n(u-u_{i,k-1}(x))-\sigma_i(x,p)S+H_i(x,u,p).
\end{eqnarray*}
For $n>\lfloor C_H \rfloor$, where $C_H$ is defined in assumption $\mathcal{P}$ (iv) c), it follows then from the comparison principle that:
\begin{eqnarray*}
w_{i,k}^-(x)\leq u_{i,k}(x)\leq w_{i,k}^+(x),~~\forall x\in[0,\kappa],
\end{eqnarray*}
and then:
\begin{eqnarray*}
\partial_xw_{i,k}^-(0)\leq \partial_x u_{i,k}(0)\leq \partial_xw_{i,k}^+(0).
\end{eqnarray*}
We look for $w^+_{i,k}$ defined on $[0,\kappa]$ of the form:
\begin{eqnarray*}
&w_{i,0}^+=g_i(x)\\
&w^+_{i,k}:x\mapsto u_{i,k}(0) +\displaystyle\frac{1}{\beta}\ln(1+\theta x),
\end{eqnarray*}
where the constants $(\beta,\theta,\kappa)$ will be chosen in the sequel independent of $k$.
Remark first that for all $x\in[0,\kappa]$, $\partial_x^2w^+_{i,k}(x)=-\beta\partial_xw^+_{i,k}(x)^2$, and $w^+_{i,k}(0)=u_{i,k}(0)$. 
Let us choose $(\theta,\kappa)$, such that: 
\begin{eqnarray}\label{eq : con upp bar 1}
\forall k\in\{1\ldots n\},~~0< \kappa\leq \min_{i\in\{1\ldots I\}}a_i,~~w^+_{i,k}(\kappa)\ge M,~~\partial_xw^+_{i,k}(\kappa)\ge \beta.
\end{eqnarray}
We choose for instance:
\begin{eqnarray}\label{eq : con upp bar 2}
\nonumber&\theta=\beta^2\exp(2\beta M)~~+~~\displaystyle\frac{1}{\min_{i\in\{1\ldots I\}}a_i}\exp(2\beta M)\\
&\kappa=\displaystyle\frac{1}{\theta}\Big(\exp(2\beta M)-1\Big).
\end{eqnarray}
The constant $\beta$ will be chosen in order to get:
\begin{eqnarray}\label{eq : con upp  bar 3}
\beta~~\ge~~\sup_{k\in\{1\ldots n\}}\sup_{x\in[0,\kappa]}\displaystyle\frac{\mu(w^+_{i,k}(x))(1+\partial_xw^+_{i,k}(x))^m+M}{\underline{\nu}(1+\partial_xw^+_{i,k}(x))^{m-2}\partial_xw^+_{i,k}(x)^2},
\end{eqnarray}
where $(\mu(.),\underline{\nu},m)$ are defined in assumption $(\mathcal{P})$ (ii) and (iii).
Since we have:
\begin{eqnarray*}
&\forall x\in[0,\kappa],~~w^+_{i,k}(x)\leq w^+_{i,k}(\kappa)=2M,\\
&\beta\leq\partial_xw^+_{i,k}(\kappa)\leq\partial_xw^+_{i,k}(x)\leq\partial_xw^+_{i,k}(0),
\end{eqnarray*}
we can then choose $\beta$ large enough to get \eqref{eq : con upp  bar 3}, for instance:
\begin{eqnarray*}
\beta~~\ge~~\displaystyle\frac{\mu(2M)}{\underline{\nu}}\Big(1+\frac{1}{\beta}\Big)^2+\displaystyle\frac{M}{\underline{\nu}\beta^2}.
\end{eqnarray*}
It is easy to show by induction that $w^+_{i,k}$ is lower barrier of $u_{i,k}$ in the neighborhood $[0,\kappa]$. More precisely, since $w^+_{i,0}=u_{i,0}$, and for all $k\in\{1\ldots n\}$: 
\begin{eqnarray*}
&w^+_{i,k}(0)=u_{i,k}(0),~~w^+_{i,k}(\kappa)\ge u_{i,k}(\kappa),\\
&w^+_{i,k}(x)=w^+_{i,k-1}(x)+u_{i,k}(0)-u_{i,k-1}(0)\ge w^+_{i,k-1}(x)-\displaystyle \frac{M}{n},
\end{eqnarray*}
we get using the assumption of induction, assumption $(\mathcal{P})$ (ii) and (iii), and \eqref{eq : con upp  bar 3} that for all $x\in(0,\kappa)$:
\begin{eqnarray*}
&n(w^+_{i,k}(x)-u_{i,k-1}(x))-\sigma_i(x,\partial_xw^+_{i,k}(x))\partial_x^2w^+_{i,k}(x)+H_i(x,w^+_{i,k}(x),\partial_xw^+_{i,k}(x))\ge \\
&-M+\beta\sigma_i(x,\partial_xw^+_{i,k}(x))\partial_{x}w^+_{i,k}(x)^2+ H_i(x,w^+_{i,k}(x),\partial_xw^+_{i,k}(x))\ge\\
&-M+\beta\underline{\nu}(1+\partial_{x}w^+_{i,k}(x))^{m-2}\partial_{x}w^+_{i,k}(x)^2+ \mu(w^+_{i,k}(x))(1+\partial_{x}w^+_{i,k}(x))^{m}\ge 0.
\end{eqnarray*}
We obtain therefore:
\begin{eqnarray*}
\sup_{n\ge 0}~~\max_{k\in\{0\ldots n\}}~~\max_{i\in\{1\ldots I\}}~~\partial_x u_{i,k}(0)~~\leq~~\frac{\theta}{\beta}\vee \max_{i\in\{1\ldots I\}} \partial_x g_i(0).
\end{eqnarray*}
With the same arguments we can show that:
\begin{eqnarray*}
&w_{i,0}^-=g_i(x)\\
&w^{-}_{i,k}:x\mapsto u_{i,k}(0) -\displaystyle\frac{1}{\beta}\ln(1+\theta x),
\end{eqnarray*}
is a lower barrier in the neighborhood of $0$. Using the same method, we can show that $\partial_xu_{i,k}(a_i)$ is uniformly bounded by the same upper bounds, which completes the proof of \textbf{Step 1}.

\textbf{Step 2 :} 
For the convenience of the reader, we do not detail all the computations of this Step, since they can be found in the proof of Theorem 15.2 of \cite{Gilbarg}. 
It follows from Lemma \ref{lm: esti derivée temps} that there exists $M>0$ such that: 
\begin{eqnarray*}
&\sup_{n\ge 0}~~\max_{k\in\{0\ldots n\}}~~\max_{i\in\{1\ldots I\}}~~\Big\{~~|u_{i,k}|_{(0,a_i)}~~\Big\}~~\leq~~M.
\end{eqnarray*}
We set furthermore:
\begin{eqnarray*}
\forall (x,u,p)\in[0,a_i]\times\R^2,~~H_{i,k}^n(x,u,p)=n(u-u_{i,k-1}(x))+H_i(x,u,p).
\end{eqnarray*}
Let $u$ be a solution of the elliptic equation, for $x\in(0,a_i)$:
\begin{eqnarray*}
\sigma_i(x,\partial_xu(x))\partial_xu(x)-H_{i,k}^n(x,u(x),\partial_xu(x))=0,
\end{eqnarray*}
and assume that $|u|_{(0,a_i)}\leq M$.
The main key of the proof will be in the use of the following equalities:
\begin{eqnarray}\label{eq :opp appl a H}
\delta_iH_{i,k}^n(x,u,p)=\delta_iH_i(x,u,p)+\frac{n(p-\partial_xu_{i,k-1}(x))}{p},~~\overline{\delta}_iH_{i,k}^n(x,u,p)=\delta_iH_i(x,u,p),
\end{eqnarray}
where we recall that the operators $\delta_i$ and $\bar \delta_i$ are defined in assumption $(\mathcal{E})$ (iii). We follow the proof of Theorem 15.2 in \cite{Gilbarg}. We set $u=\psi(\overline{u})$, where $\psi\in\mathcal{C}^3[\overline{m},\overline{M}]$, is increasing and $\overline{m}=\phi(-M)$, $\overline{M}=\phi(M)$.

In the sequel, we will set $v=\partial_x u^2$ and $\overline{v}=\partial_x\overline{u}^2$. To simplify the notations, we will omit the variables $(x,u(x),\partial_xu(x))$ in the functions $\sigma_i$ and $H_{i,k}^n$, and the variable $\overline{u}$ for $\psi$. 
We assume first that the solution $u\in\mathcal{C}^3([-M,M])$, and we follow exactly all the computations that lead to equation of (15.25) of \cite{Gilbarg} to get the following inequality:
\begin{eqnarray}\label{eq : ineq prin maxim}
\sigma_i\partial_x^2\overline{v}+B_i\partial_x\overline{v}+G_{i,k}^n\ge 0,
\end{eqnarray}
where $B_i$ and $G_{i,k}^n$ have the same expression in (15.26) of \cite{Gilbarg} with $(\sigma_i=\sigma_i^*,c_i=0$). We choose $(r=0,s=0)$, since we will see in the sequel \eqref{eq : condit de bornitud coeff}, that condition (15.32) of \cite{Gilbarg} holds under assumption assumption $(\mathcal{P})$. We have more precisely: 
\begin{eqnarray*}
&B_i=\psi'\partial_p\sigma_i\partial_x\overline{u}-\partial_pH_i+\omega\partial_p(\sigma_ip^2),\\
&G_{i,k}^n=\displaystyle\frac{\omega'}{\psi'}+\kappa_i\omega ^2+\beta_i\omega+\theta_{i,k}^n,\\
&\omega=\displaystyle\frac{\psi''}{\psi'^2}\in\mathcal{C}^1([\overline{m},\overline {M}]),\\
&\kappa_i=\displaystyle\frac{1}{\sigma_ip^2}\Big(\overline{\delta}_i(\sigma_ip^2)+\frac{p^2}{4\sigma_i}|(\overline{\delta}_i+1)\sigma_i|^2\Big),\\
&\beta_i=\displaystyle\frac{1}{\sigma_ip^2}\Big(\delta_i(\sigma_ip^2)-\overline{\delta}_iH_i+\frac{p^2}{2\sigma_i}((\overline{\delta}_i+1)\sigma_i)(\delta_i\sigma_i)\Big),
\\&\theta_{i,k}^n=\displaystyle\frac{1}{\sigma_ip^2}\Big(\frac{p^2}{4\sigma_i}|\delta_i\sigma_i|^2-\delta_iH_{i,k}^n\Big)=\theta_i-\displaystyle\frac{1}{\sigma_ip^2}\Big(\frac{n(p-\partial_xu_{i,k-1}(x))}{p}\Big),\\
&\theta_i=\displaystyle\frac{1}{\sigma_ip^2}\Big(\frac{p^2}{4\sigma_i}|\delta_i\sigma_i|^2-\delta_iH_i\Big).
\end{eqnarray*}
We set in the sequel:
\begin{eqnarray*}
G_i=\displaystyle\frac{\partial_x\omega}{\partial_x\psi}+\kappa_i\omega ^2+\beta_i\omega+\theta_{i},\text{ in order to get } G_{i,k}^n=G_i-\displaystyle\frac{1}{\sigma_ip^2}\Big(\frac{n(p-\partial_xu_{i,k-1}(x))}{p}\Big).
\end{eqnarray*}
More precisely, we see from \eqref{eq :opp appl a H} that all the coefficients $(B_i,\kappa_i,\beta_i,\theta_i)$ can be chosen independent of $n$ and $u_{i,k-1}$. The main argument then to get a bound of $\partial_xu$ is to apply a maximum principle for  $\overline{v}$ in \eqref{eq : ineq prin maxim}, and this will be done as soon as we ensure:
\begin{eqnarray*}
G_{i,k}^n\leq 0, \text{ for }|\partial_xu|\ge L_k^n.
\end{eqnarray*}
On the other hand, using assumption ($\mathcal{P}$) (ii) (iii) and (iv), it is easy to check that there exist constants $(a,b,c)$, depending only on the data: 
\begin{eqnarray*}
\Big(\overline{\nu},\underline{\nu},\mu(M),\gamma(M),\varepsilon(M),\sup_{|p|\ge 0}P(M,|p|)\Big),
\end{eqnarray*}
such that:
\begin{eqnarray} \label{eq : condit de bornitud coeff}
\nonumber&\sup_{x\in[0,a_i],|u|\leq M}~~\limsup_{|p|\to +\infty }~~\kappa_i(x,u,p)~~\leq~~a,\\\nonumber&
\sup_{x\in[0,a_i],|u|\leq M}~~\limsup_{|p|\to +\infty }~~\beta_i(x,u,p)~~\leq~~b,\\
&\sup_{x\in[0,a_i],|u|\leq M}~~\limsup_{|p|\to +\infty }~~\theta_i(x,u,p)~~\leq~~c,
\end{eqnarray}
where, 
\begin{eqnarray*}
&a=\displaystyle\frac{1}{\underline{\nu}}(\gamma(M)+2\overline{\nu})+\frac{1}{2}+\frac{\gamma(M)^2}{\underline{\nu}^2},\\
&b=\displaystyle\frac{\varepsilon(M)+\sup_{|p|\ge 0}P(M,|p|)+\gamma(M)}{\underline{\nu}}+\displaystyle\frac{(\varepsilon(M)+\sup_{|p|\ge 0}P(M,|p|))(\overline{\nu}+\gamma(M))}{\underline{\nu}^2},\\
&c=\displaystyle\frac{(\varepsilon(M)+\sup_{|p|\ge 0}P(M,|p|))^2}{4\underline{\nu}^2}+\displaystyle\frac{2(\varepsilon(M)+\sup_{|p|\ge 0}P(M,|p|))}{\underline{\nu}}.
\end{eqnarray*}
As it has been on the proof of Theorem 15.2 of \cite{Gilbarg}, we choose then $L=L(a,b,c)$, and $\psi(\cdot)=\psi(a,b,c)(\cdot)$ such that we have:
\begin{eqnarray*}
G_i\leq 0,~~\text{if}~~|\partial_xu(x)|\ge L(a,b,c).
\end{eqnarray*}
We see then from the expression of $\theta_{i,k}^n$ that we get
\begin{eqnarray*}
G_{i,k}^n\leq 0,~~\text{if}~~|\partial_xu(x)|\ge L(a,b,c)\vee |\partial_xu_{i,k-1}(x)|.
\end{eqnarray*}
Therefore applying the maximum principle to $\overline{v}$ in \eqref{eq : ineq prin maxim}, and from the relation $u=\psi(\overline{u})$, $\overline{v}=\partial_x\overline{u}^2$ we get finally: 
\begin{eqnarray*}
|\partial_xu|_{(0,a_i)}\leq \max\Big(\frac{\max \psi'(a,b,c)(\cdot)}{\min\psi'(a,b,c)(\cdot)},|\partial_xu|_{\partial(0,a_i)},L(a,b,c),|\partial_xu_{i,k-1}|_{(0,a_i)}\Big).
\end{eqnarray*}
This upper bound still holds if $u\in\mathcal{C}^2([0,a_i])$, (cf. (15.30) and (15.31) of the proof of Theorem 15.2 in \cite{Gilbarg}). Finally applying the upper bound above to the solution $u_k$, we get by induction that:
\begin{eqnarray*}
&\sup_{n\ge 0}~~\max_{k\in\{0\ldots n\}}~~\max_{i\in\{1\ldots I\}}~~\Big\{~~|\partial_xu_{i,k}|_{(0,a_i)}~~\Big\}\\
&\leq~~\max_{i\in\{1\ldots I\}}\max\Big(\displaystyle\frac{\max \psi'(a,b,c)(\cdot)}{\min\psi'(a,b,c)(\cdot)},|\partial_xu_{i,k}|_{\partial(0,a_i)},L(a,b,c),|\partial_xg_{i}|_{(0,a_i)}\Big).
\end{eqnarray*}
This completes the proof.
\end{proof}
The following Proposition follows from Lemmas \ref{lm: esti derivée temps} and \ref{lm : borne de grad u}, assumption $(\mathcal{P})$ (ii) (iii), and from the relation: 
\begin{eqnarray*}
&\forall x\in [0,a_i],~~|\partial_x^2u_{i,k}(x))|\leq \displaystyle\frac{|n(u_{i,k}(x)-u_{i,k-1}(x))|+|H_{i}(x,u_{i,k}(x),\partial_xu_{i,k}(x))|}{\sigma_i(x,\partial_xu_{i,k}(x))}\\
&\leq \displaystyle\frac{|n(u_{i,k}(x)-u_{i,k-1}(x))|+\mu(|u_{i,k}(x)|)(1+|\partial_xu_{i,k}(x)|^{m})}{\underline{\nu}(1+|\partial_xu_{i,k}(x)|^{m-2})}.
\end{eqnarray*}
\begin{Proposition} \label{pr: esti pas} 
Assume ($\mathcal{P}$). There exist constants $(M_1,M_2,M_3)$, depending only the data introduced in assumption $(\mathcal{P})$ 
\begin{eqnarray*}
&M_1=M_1\Big(\max_{i\in\{1\ldots I\}}\Big\{~~\sup_{x\in(0,a_i)}|-\sigma_i(x,\partial_xg_i(x))\partial_x^2g_i(x)+H_{i}(x,g_i(x),\partial_xg_i(x))|+\\
&|\partial_t\phi_i|_{(0,T)}~~\Big\},\max_{i\in\{1\ldots I\}} |g_i|_{(0,a_i)},C_H\Big),\\
&M_2=M_2\Big(\overline{\nu},\underline{\nu},\mu(M_1),\gamma(M_1),\varepsilon(M_1),\sup_{|p|\ge 0}P(M_1,|p|),\max_{i\in\{1\ldots I\}}|\partial_xg_i|_{(0,a_i)},M_1\Big),
\\
&M_3=M_3\Big(M_1,\underline{\nu}(1+|p|)^{m-2},\mu(|u|)(1+|p|)^m,~~|u|\leq M_1,~~|p|\leq M_2\Big),\\
\end{eqnarray*}
such that:
\begin{eqnarray*}
\sup_{n\ge 0}~~\max_{k\in\{0\ldots n\}}~~\max_{i\in\{1\ldots I\}}~~\Big\{~~|u_{i,k}|_{(0,a_i)}~~\Big\}~~&\leq~~M_1,\\
\sup_{n\ge 0}~~\max_{k\in\{0\ldots n\}}~~\max_{i\in\{1\ldots I\}}~~\Big\{~~|\partial_xu_{i,k}|_{(0,a_i)}~~\Big\}~~&\leq~~M_2,\\
\sup_{n\ge 0}~~\max_{k\in\{1\ldots n\}}~~\max_{i\in\{1\ldots I\}}~~\Big\{~~|n(u_{i,k}-u_{i,k-1})|_{(0,a_i)}~~\Big\}~~&\leq~~M_1,\\
\sup_{n\ge 0}~~\max_{k\in\{0\ldots n\}}~~\max_{i\in\{1\ldots I\}}~~\Big\{~~|\partial_x^2u_{i,k}|_{(0,a_i)}~~\Big\}~~&\leq~~M_3.
\end{eqnarray*}
\end{Proposition}
Unfortunately, we are unable to give an upper bound of the modulus of continuity of $\partial_x^2u_{i,k}$ in $\mathcal{C}^{\alpha}([0,a])$ independent of $n$. However, we are able to formulate in the weak sense a limit solution. From the regularity of the coefficients, using some tools introduced in Section \ref{sec intro}, Lemma \ref{lm : cont deruiv temps}, we get interior regularity, and a smooth limit solution.
\subsection{Proof of Theorem \ref{th : exis para}.}
\begin{proof} 
The uniqueness is a result of the comparison Theorem \ref{th : para comparison th}. To simplify the notations, we set for each $i\in\{1\ldots I\}$, and for each $(x,q,u,p,S)\in [0,a_i]\times R^4$
\begin{eqnarray*}
Q_{i}(x,u,q,p,S)~~=~~q-\sigma_i(x,p)S+H_i(x,u,p).
\end{eqnarray*}
Let $n\ge 0$. Consider the subdivision $(t_k^n=\frac{kT}{n})_{0\leq k\leq n}$ of $[0,T]$, and $(u_{k})_{0\leq k\leq n}$ the solution of \eqref{eq: schema ell}.

From estimates of Proposition \ref{pr: esti pas}, there exists a constant $M>0$ independent of $n$, such that:  
\begin{eqnarray}\label{eq : bornes glob}
&\nonumber\sup_{n\ge 0}~~\max_{k\in\{1\ldots n\}}~~\max_{i\in\{1\ldots I\}}\Big\{~~|u_{i,k}|_{(0,a_i)}~~+~~|n(u_{i,k}-u_{i,k-1})|_{(0,a_i)}~~+\\&
~~|\partial_xu_{i,k}|_{(0,a_i)}~~+~~|\partial_x^2u_{i,k}|_{(0,a_i)}~~\Big\}~~\leq~~M.
\end{eqnarray}
We define the following sequence $(v_{n})_{n\ge0}$ in $\mathcal{C}^{0,2}(\mathcal{J}_T^a)$, piecewise differentiable with respect to its first variable by: 
\begin{eqnarray*}
&\forall i\in \{1\ldots I\}, ~~v_{i,0}(0,x)=g_i(x)~~ \text{ if } x\in [0,a_i],\\
&v_{i,n}(t,x)~~=~~u_{i,k}(x)+n(t-t_k^n)(u_{i,k+1}(x)-u_{i,k}(x))~~\text{ if } (t,x)\in [t_k^n,t_{k+1}^n)\times[0,a_i].
\end{eqnarray*}
We deduce then from \eqref{eq : bornes glob}, that there exists a constant $M_1$ independent of $n$, depending only on the data of the system, such that for all $i\in\{1\ldots I\}$
\begin{eqnarray*}
|v_{i,n}|_{[0,T]\times[0,a_i]}^{\alpha}~~+~~|\partial_xv_{i,n}|_{x,[0,T]\times[0,a_i]}^{\alpha}~~\leq~~ M_1.
\end{eqnarray*}
Using Lemma \ref{lm : cont deruiv temps}, we deduce that there exists a constant $M_2(\alpha)>0$, independent of $n$, such that for all $i\in\{1\ldots I\}$, we have the following global H\"{o}lder condition:
\begin{eqnarray*}
|\partial_xv_{i,n}|_{t,[0,T]\times[0,a_i]}^{\frac{\alpha}{2}}~~+~~|\partial_xv_{i,n}|_{x,[0,T]\times[0,a_i]}^{\alpha}~~\leq ~~M_2(\alpha).
\end{eqnarray*}
We deduce then from Ascoli's Theorem, that up to a sub sequence $n$, $(v_{i,n})_{n\ge 0}$ converge in $\mathcal{C}^{0,1}([0,T]\times[0,a_i])$ to $v_i$, and then $v_i\in \mathcal{C}^{\frac{\alpha}{2},1+\alpha}([0,T]\times[0,a_i])$.
Since $v_{n}$ satisfies the following continuity condition at the junction point 
\begin{eqnarray*}
\forall (i,j)\in \{1\ldots I\}^2,~~\forall n\ge 0,~~\forall t\in[0,T],~~v_{i,n}(t,0)~~=~~v_{j,n}(t,0),
\end{eqnarray*}
we deduce then $v\in \mathcal{C}^{\frac{\alpha}{2},1+\alpha}(\mathcal{J}^a_T)$.

We now focus on the regularity of $v$ in $\overset{\circ}{\mathcal{J}^a_T}$, and we will prove that $v\in \mathcal{C}^{1+\frac{\alpha}{2},2+\alpha}(\overset{\circ}{\mathcal{J}^a_T})$, and satisfies on each edge: 
\begin{eqnarray*}
Q_{i}(x,v_i(t,x),\partial_tv_i(t,x),\partial_xv_i(t,x),\partial_x^2v_i(t,x))~~=~~0,~~ \text{ if } (t,x) \in (0,T)\times (0,a_i).
\end{eqnarray*}
Using once again \eqref{eq : bornes glob}, there exists a constant $M_3$ independent of $n$, such that for each $i\in\{1\ldots I\}$:
\begin{eqnarray*}
\|\partial_{t}v_{i,n}\|_{L^{2}((0,T)\times(0,a_i))}~~\leq~~ M_3,~~\|\partial_x^2v_{i,n}\|_{L^{2}((0,T)\times(0,a_i))}~~\leq~~ M_3.
\end{eqnarray*}
Hence we get up to a sub sequence, that:
\begin{eqnarray*}
\partial_{t}v_{i,n}~~{\rightharpoonup}~~\partial_{t}v_i,~~\partial_x^2v_{i,n}~~{\rightharpoonup}~~\partial_x^2v_i,
\end{eqnarray*}
weakly in $L^2((0,T)\times(0,a_i))$.

The continuity of the coefficients $(\sigma_i,H_i)_{i\in\{1 \ldots I\}}$, Lebesgue Theorem, the linearity of $Q_i$ in the variable $\partial_t$ and $\partial_x^2$, allows us to get for each $i\in\{1 \ldots I\}$, up to a subsequence $n_p$:
\begin{eqnarray*}
&\displaystyle\int_0^T\int_0^{a_i}\Big(Q_{i}(x,v_{i,n_p}(t,x),\partial_tv_{i,n_p}(t,x),\partial_xv_{i,n_p}(t,x),\partial_x^2v_{i,n_p}(t,x))\Big)\psi(t,x)dxdt\\
&\xrightarrow[]{p\to +\infty}~~\displaystyle\int_0^T\int_0^{a_i}\Big(Q_{i}(x,v_i(t,x),\partial_tv_i(t,x),\partial_xv_i(t,x),\partial_x^2v_i(t,x))\Big)\psi(t,x)dxdt,\\
&~~\forall \psi\in\mathcal{C}_c^\infty((0,T)\times (0,a_i)).
\end{eqnarray*}
We now prove that for any $\psi\in\mathcal{C}_c^\infty((0,T)\times (0,a_i))$:
\begin{eqnarray*}
&\displaystyle\int_0^T\int_0^{a_i}\Big(Q_{i}(x,v_{i,n_p}(t,x),\partial_tv_{i,n_p}(t,x),\partial_xv_{i,n_p}(t,x),\partial_x^2v_{i,n_p}(t,x))\Big)\psi(t,x)dxdt
~~\xrightarrow[]{p\to +\infty}~~0.
\end{eqnarray*}
Using that $(u_{k})_{0\leq k\leq n}$ is the solution of \eqref{eq: schema ell}, we get for any $\psi\in\mathcal{C}_c^\infty((0,T)\times (0,a_i))$:
\begin{eqnarray} \label{eq : eq 1}
\nonumber &\displaystyle\int_0^T\int_0^{a_i}\Big(Q_{i}(x,v_{i,n}(t,x),\partial_tv_{i,n}(t,x),\partial_xv_{i,n}(t,x),\partial_x^2v_{i,n}(t,x))\Big)\psi(t,x)dxdt~~=\\\nonumber 
&\displaystyle\sum_{k=0}^{n-1}~~\displaystyle\int_{t_k^n}^{t_{k+1}^n}\int_0^{a_i}\Big(~\sigma_i(x,\partial_xu_{i,k+1}(x))\partial_x^2u_{i,k+1}(x)-\sigma_i(x,\partial_xv_{i,n}(t,x))\partial_x^2v_{i,n}(t,x)
\\&+H_i(x,v_{i,n}(t,x),\partial_xv_{i,n}(t,x))-H_i(x,u_{i,k+1}(x),\partial_xu_{i,k+1}(x))~\Big)\psi(t,x)dxdt.
\end{eqnarray}
Using assumption $(\mathcal{P})$ more precisely the Lipschitz continuity of the Hamiltonians $H_i$, the H\"{o}lder equicontinuity in time of $(v_{i,n},\partial_xv_{i,n})$, there exists a constant $M_4(\alpha)$ independent of $n$, such that for each $i\in \{1\ldots I\}$, for each $(t,x)\in [t_k^n,t_{k+1}^n]\times[0,a_i]$:
\begin{eqnarray*}
|H_i(x,u_{i,k+1}(x),\partial_xu_{i,k+1}(x))
-H_i(x,v_{i,n}(t,x),\partial_xv_{i,n}(t,x))|&\leq&M_4(\alpha)(t-t_k^n)^{\frac{\alpha}{2}}, 
\end{eqnarray*}
and therefore for any $\psi\in\mathcal{C}_c^\infty((0,T)\times (0,a_i))$:
\begin{eqnarray*}
&\Big|~~\displaystyle\sum_{k=0}^{n-1}~~\displaystyle\int_{t_k^n}^{t_{k+1}^n}\int_0^{a_i}\Big(H_i(x,u_{i,k+1}(x),\partial_xu_{i,k+1}(x))
-H_i(x,v_{i,n}(t,x),\partial_xv_{i,n}(t,x))\Big)\psi(t,x)dxdt~~\Big|\leq\\
&a_iM_4(\alpha)|\psi|_{(0,T)\times(0,a_i)}n^{-\frac{\alpha}{2}}~~\xrightarrow[]{n\to +\infty}~~0.
\end{eqnarray*}
For the last term in \eqref{eq : eq 1}, we write  for each $i\in \{1\ldots I\}$, for each $(t,x)\in (t_k^n,t_{k+1}^n)\times(0,a_i)$:
\begin{eqnarray}
\nonumber &\sigma_i(x,\partial_xu_{i,k+1}(x))\partial_x^2u_{i,k+1}(x)-\sigma_i(x,\partial_xv_{i,n}(t,x))\partial_x^2v_{i,n}(t,x)~~=\\~~\label{eq : 2}&
\Big(\sigma_i(x,\partial_xu_{i,k+1}(x))-\sigma_i(x,\partial_xv_{i,n}(t,x))\Big)\partial_x^2u_{i,k}(x)+\\\label{eq : 3}&\Big(\sigma_i(x,\partial_xu_{i,k+1}(x))-n(t-t_k^n)\sigma_i(x,\partial_xv_{i,n}(t,x))\Big)\Big(\partial_x^2u_{i,k+1}(x)-\partial_x^2u_{i,k}(x)\Big).
\end{eqnarray}
Using again the  H\"{o}lder equicontinuity in time of $(v_{i,n},\partial_xv_{i,n})$ as well as the uniform bound on $|\partial_x^2u_{i,k}|_{[0,a_i]}$ \eqref{eq : bornes glob}, we can show that for \eqref{eq : 2}, for any $\psi\in\mathcal{C}_c^\infty((0,T)\times (0,a_i))$:
\begin{eqnarray*}
\Big|~~\displaystyle\sum_{k=0}^{n-1}~~\displaystyle\int_{t_k^n}^{t_{k+1}^n}\int_0^{a_i}\Big(\sigma_i(x,\partial_xu_{i,k+1}(x))-\sigma_i(x,\partial_xv_{i,n}(t,x))\Big)\partial_x^2u_{i,k}(x)\psi(t,x)dxdt~~\Big|
~~\xrightarrow[]{n\to +\infty}~~0.
\end{eqnarray*}
Finally, from assumptions $(\mathcal{P})$, for all $i\in \{1\ldots I\}$, $\sigma_i$ is differentiable with respect to all its variable, integrating by part we get for \eqref{eq : 3}:
\begin{eqnarray*}
&\Big|\displaystyle\sum_{k=0}^{n-1}\displaystyle\int_{t_k^n}^{t_{k+1}^n}\int_0^{a_i}\Big(\sigma_i(x,\partial_xu_{i,k+1}(x))-n(t-t_k^n)\sigma_i(x,\partial_xv_{i,n}(t,x))\Big)\\
&\Big(\partial_x^2u_{i,k+1}(x)-\partial_x^2u_{i,k}(x)\Big)
\psi(t,x)dxdt\Big|=\\
&\Big|\displaystyle\sum_{k=0}^{n-1}\displaystyle\int_{t_k^n}^{t_{k+1}^n}\int_0^{a_i}\Big(\partial_x\Big(\sigma_i(x,\partial_xu_{i,k+1}(t,x))\psi(t,x)\Big)-n(t-t_k^n)\partial_x\Big(\sigma_i(x,\partial_xv_{i,n}(t,x))\psi(t,x)\Big)\Big)\\&
\Big(\partial_xu_{i,k+1}(x)-\partial_xu_{i,k}(x)\Big)dxdt\Big|~~\xrightarrow[]{n\to +\infty}~~0.
\end{eqnarray*}
We conclude that for any $\psi\in\mathcal{C}_c^\infty((0,T)\times (0,a_i))$:
\begin{eqnarray*}
~~\displaystyle\int_0^T\int_0^{a_i}\Big(Q_{i}(x,v_i(t,x),\partial_tv_i(t,x),\partial_xv_i(t,x),\partial_x^2v_i(t,x)))\Big)\psi(t,x)dxdt~~=0.
\end{eqnarray*}
It is then possible to consider the last equation as a linear one, with coefficients $\tilde{\sigma_i}(t,x)=\sigma_i(x,\partial_xv_i(t,x))$, $\tilde{H}_i(t,x)=H_i(x,v_{i}(t,x),\partial_xv_{i}(t,x))$ belonging to the class $\mathcal{C}^{\frac{\alpha}{2},\alpha}((0,T)\times(0,a_i))$, and using Theorem III.12.2 of \cite{pde para}, we get finally that for all $i\in\{1\ldots I\}$, $v_i\in \mathcal{C}^{1+\frac{\alpha}{2},2+\alpha}((0,T)\times(0,a_i))$, which means that $v\in \mathcal{C}^{1+\frac{\alpha}{2},2+\alpha}(\overset{\circ}{\mathcal{J}_T^a}) $.

We deduce that $v_i$ satisfies on each edge:
\begin{eqnarray*}
Q_{i}(x,v_i(t,x),\partial_tv_i(t,x),\partial_xv_i(t,x),\partial_x^2v_i(t,x)))~~=~~0,~~ \text{ if } (t,x) \in (0,T)\times (0,a_i).
\end{eqnarray*}
From the estimates \eqref{eq : bornes glob}, we know that $\partial_tv_{i,n}$ and $\partial_x^2v_{i,n}$ are uniformly bounded by $n$. We deduce finally that $v\in \mathcal{C}_b^{1+\frac{\alpha}{2},2+\alpha}(\overset{\circ}{\mathcal{J}_T^a})$.

We conclude by proving that $v$ satisfies the non linear Neumann boundary condition at the vertex. For this, let $t\in(0,T)$; we have up to a sub sequence $n_p$:
\begin{eqnarray*}
F(v_{n_p}(t,0),\partial_x v_{n_p}(t,0))~~\xrightarrow[p\to +\infty]{}~~F(v(t,0),\partial_x v(t,0)).&
\end{eqnarray*}
On the other hand, using that $F(u_k(0),\partial_0 u_k(x))=0$, we know from the continuity of $F$ (assumption $(\mathcal{P})$), the H\"{o}lder equicontinuity in time of $t\mapsto v_n(t,0)$, and $t\mapsto\partial_x v(t,0)$, that there exists a constant $M_5(\alpha)$ independent of $n$, such that if $t\in[t_k^n,t_{k+1}^n)$:
\begin{eqnarray*}
&|F(v_n(t,0),\partial_x v_n(t,0))|=|F(v_n(t,0),\partial_x v_n(t,0))-F(u_k(0),\partial_x u_k(0))|\leq\\&
\sup\Big\{|F(u,x)-F(v,y)|,~~\|u-v\|_{\R^I}+\|x-y\|_{\R^I}\leq M_5(\alpha)n^{-\frac{\alpha}{2}}\Big\}\xrightarrow[]{n\to +\infty}0.
\end{eqnarray*}
Therefore, we conclude once more from the continuity of $F$ (assumption $(\mathcal{P})$), the compatibility condition (assumption $(\mathcal{P})$ (v)), that for each $t\in[0,T)$:
\begin{eqnarray*}
F(v(t,0),\partial_x v(t,0))~~=~~0.
\end{eqnarray*}
On the other hand, it is easy to get: 
\begin{eqnarray*}
\forall i\in\{1\ldots I\},~~\forall x\in [0,a_i],~~~v_i(0,x)=g_i(x),~~\forall t\in[0,T],~~~v_i(t,a_i)=\phi_i(t).
\end{eqnarray*}
Finally, the expression of the upper bounds of the solution given in Theorem \ref{th : exis para}, are a consequence of Proposition \ref{pr: esti pas}, and Lemma \ref{lm : cont deruiv temps}, which completes the proof.
\end{proof}
\subsection{On the existence for unbounded junction}
We give in this subsection a result on the existence and the uniqueness of the solution for the parabolic problem \eqref{eq de base}, posed on an unbounded junction $\mathcal{J}$ defined for $I\in \mathbb{N}^*$ edges by:
\begin{eqnarray*}\mathcal{J}=\bigcup_{i=1}
^IJ_i,\text{  with:}~~\forall i\in \{1\ldots I\}~~J_i=[0,+\infty),~~\text{and}~~\forall (i,j)\in\{1\ldots I\}^2,~~i\neq j,~~J_i\cap J_i=\{0\}. \end{eqnarray*}
In the sequel, $\mathcal{C}^{0,1}(\mathcal{J}_T)\cap \mathcal{C}^{1,2}(\overset{\circ}{\mathcal{J}_T})$ is the class of function with regularity $\mathcal{C}^{0,1}([0,T]\times[0,+\infty))\cap \mathcal{C}^{1,2}((0,T)\times(0,+\infty))$ on each edge, and $L^{\infty}(\mathcal{J}_T)$ is the set of measurable real bounded maps defined on $\mathcal{J}_T$.

We introduce the following data
$$\begin{cases}
F\in \mathcal{C}^0(\R^I\times\R^I,\R)\\
g\in\mathcal{C}^{1}_b(\mathcal{J})\cap\mathcal{C}_b^{2}(\overset{\circ}{\mathcal{J}})
\end{cases},$$
and for each $i\in\{1\ldots I\}$
$$\begin{cases}
\sigma_i \in \mathcal{C}^{1}(\R_{+}\times\R,\R)\\
H_i \in \mathcal{C}^1(\R_{+}\times\R^2,\R)\\ 
\phi_i \in \mathcal{C}^1([0,T],\R)
\end{cases}.$$
We suppose furthermore that the data satisfy the following assumption
$$\textbf{Assumption } (\mathcal{P}_{\infty})$$
(i) Assumption on $F$:
$$\begin{cases}
a)~~F \text{ is decreasing with respect to its first variable,} \\
b)~~F \text{ is nondecreasing with respect to its second variable,}\\
c)~~\exists(b,B)\in\R^I\times\R^I,~~F(b,B)=0,
\end{cases}$$
or the Kirchhoff condition:
$$\begin{cases}
a)~~F \text{ is nonincreasing with respect to its first variable,} \\
b)~~F \text{ is increasing with respect to its second variable,}\\
c)~~\exists(b,B)\in\R^I\times\R^I,~~F(b,B)=0.
\end{cases}$$
We suppose moreover that there exists a parameter $m\in \R$, $m\ge 2$ such that we have\\
(ii) The (uniform) ellipticity condition on the $(\sigma_i)_{i\in\{1\ldots I\}}$ : there exist $\underline{\nu},\overline{\nu}$, strictly positive constants such that:
\begin{eqnarray*}
\forall i\in \{1\ldots I\},~~\forall (x,p)\in \R_{+}\times\R,\\
\underline{\nu}(1+|p|)^{m-2}~~\leq~~\sigma_i(x,p)~~\leq~~\overline{\nu}(1+|p|)^{m-2}.
\end{eqnarray*}
(iii) The growth of the $(H_i)_{i\in\{1\ldots I\}}$ with respect to $p$ exceed the growth of the $\sigma_i$ with respect to $p$ by no more than two, namely there exists $\mu$ an increasing real continuous function such that:
\begin{eqnarray*}
\forall i\in \{1\ldots I\},~~\forall (x,u,p)\in \R_{+}\times\R^2,~~|H_i(x,u,p)|~~\leq~\mu(|u|)(1+|p|)^{m}.
\end{eqnarray*}
(iv) We impose the following restrictions on the growth with respect to $p$ of the derivatives for the coefficients $(\sigma_i,H_i)_{i\in\{1\ldots I\}}$, which are for all $i\in \{1\ldots I\}$:
\begin{eqnarray*}
&a)~~|\partial_p\sigma_i|_{\R_{+}\times\R^2}(1+|p|)^2+|\partial_pH_i|_{\R_{+}\times\R^2}~~\leq~~\gamma(|u|)(1+|p|)^{m-1},\\
&b)~~|\partial_x\sigma_i|_{\R_{+}\times\R^2}(1+|p|)^2+|\partial_xH_i|_{\R_{+}\times\R^2}~~\leq~~\Big(\varepsilon(|u|)+P(|u|,|p|)\Big)(1+|p|)^{m+1},\\
&c)~~\forall (x,u,p)\in \R_{+}\times\R^2,~~-C_H~~\leq~~\partial_uH_i(x,u,p)~~\leq~~\Big(\varepsilon(|u|)+P(|u|,|p|)\Big)(1+|p|)^{m},
\end{eqnarray*}
where $\gamma$ and $\varepsilon$ are continuous non negative increasing functions. $P$ is a continuous function, increasing with respect to its first variable, and tends to $0$ for $p\to +\infty$, uniformly with respect to its first variable, from $[0,u_1]$ with $u_1\in R$, and $C_H>0$ is real strictly positive number. We assume that $(\gamma,\varepsilon,P,C_H)$ are independent of $i\in \{1\ldots I\}$.\\
(v) A compatibility conditions for $g$:
\begin{eqnarray*}
&F(g(0),\partial_xg(0))~~=~~0.
\end{eqnarray*}

We state here a comparison Theorem for the problem \ref{eq de base}, posed on an unbounded junction.
\begin{Theorem}\label{th : para comparison th jonct non bornée} 
Assume $(\mathcal{P}_{\infty})$. Let $u\in \mathcal{C}^{0,1}(\mathcal{J}_T)\cap \mathcal{C}^{1,2}(\overset{\circ}{\mathcal{J}_T})\cap L^{\infty}(\mathcal{J}_T)$ (resp. $v\in \mathcal{C}^{0,1}(\mathcal{J}_T)\cap \mathcal{C}^{1,2}(\overset{\circ}{\mathcal{J}_T})\cap L^{\infty}(\mathcal{J}_T)$) be a super solution (resp. a sub solution) of \eqref{eq : Neumann par } (where $a_i=+\infty$), satisfying for all $i\in\{1\ldots I\}$ for all $x\in [0,+\infty)$, $u_i(0,x)\ge v_i(0,x)$.
Then for each $(t,(x,i))\in\mathcal{J}_T$ : $u_i(t,x)\ge v_i(t,x)$.
\end{Theorem}
\begin{proof}
Let $s\in[0,T)$, $K=(K\ldots K)>(1,\ldots 1)$ in $\R^I$, and $\lambda=\lambda(K)>0$, that will be chosen in the sequel. We argue as in the proof of Theorem \ref{th : para comparison th}, assuming 
\begin{eqnarray*}
\sup_{(t,(x,i))\in \mathcal{J}^K_{s}} \exp(-\lambda t-\frac{(x-1)^2}{2})\Big(v_{i}(t,x)-u_{i}(t,x)\Big)>0.
\end{eqnarray*}
Using the boundary conditions satisfied by $u$ and $v$, the above supremum  is reached at a point $(t_0,(x_0,j_0))\in (0,s]\times \mathcal{J}$, with $0\leq x_0\leq K$.

If $x_0\in [0,K)$, the optimality conditions are given for $x_0\neq 0$ by:
\begin{eqnarray*}
&-\lambda(v_{j_0}(t_0,x_0)-u_{j_0}(t_0,x_0))~~+~~\partial_tv_{j_0}(t_0,x_0)-\partial_tu_{j_0}(t_0,x_0)~~\ge~~0,\\
&-(x_0-1)\Big(v_{j_0}(t_0,x_0)-u_{j_0}(t_0,x_0)\Big)~~+~~\partial_xv_{j_0}(t_0,x_0)-\partial_xu_{j_0}(t_0,x_0)~~=~~0,\\
&
\Big(v_{j_0}(t_0,x_0)-u_{j_0}(t_0,x_0)\Big)~~-~~2(x_0-1)^2\Big(v_{j_0}(t_0,x_0)-u_{j_0}(t_0,x_0)\Big)\\
&+~~\Big(\partial_x^2v_{j_0}(t_0,x_0)-\partial_x^2u_{j_0}(t_0,x_0)\Big)~~\leq 0,
\end{eqnarray*}
and if $x_0=0$:
\begin{eqnarray*}
\forall i\in\{1,\ldots I\},~~\partial_xv_i(t_0,0) ~~\leq ~~\partial_xu_i(t_0,0)-\Big(v_i(t_0,0)-u_i(t_0,0)\Big)~~<~~\partial_xu_i(t_0,0).
\end{eqnarray*}
If $x_0=0$, we obtain a contradiction exactly as in  the proof of Theorem \ref{th : para comparison th}. On the other hand if $x_0\in(0,K)$, using assumptions $(\mathcal{P})$ (iv) a), (iv) c) and the optimality conditions, we can choose $\lambda(K)$ of the form $\lambda(K)=C(1+K^2)$, (see \eqref{eq : 1cont comp} and \eqref{eq : 2const compa}), where $C>0$ is a constant independent of $K$, to get again a contradiction. 
We deduce that, if:
\begin{eqnarray*}
\sup_{(t,(x,i))\in \mathcal{J}^K_{s}} \exp(-\lambda(K)t-\frac{(x-1)^2}{2})\Big(v_{i}(t,x)-u_{i}(t,x)\Big)>0,
\end{eqnarray*}
then for all $(t,(x,i))\in[0,T]\times\mathcal{J}^K$:
\begin{eqnarray*}
\exp(-\lambda(K) t-\frac{(x-1)^2}{2})\Big(v_{i}(t,x)-u_{i}(t,x)\Big)\leq \exp(-\lambda(K) t-\frac{(K-1)^2}{2})\Big(v_{i}(t,K)-u_{i}(t,K)\Big).
\end{eqnarray*}
Hence for all $(t,(x,i))\in[0,T]\times\mathcal{J}^K$:
\begin{eqnarray*}
\exp(-\frac{(x-1)^2}{2})\Big(v_{i}(t,x)-u_{i}(t,x)\Big)\leq \exp(-\frac{(K-1)^2}{2})\Big(v_{i}(t,K)-u_{i}(t,K)\Big).
\end{eqnarray*}
On the other hand, if:
\begin{eqnarray*}
\sup_{(t,(x,i))\in \mathcal{J}^K_{s}} \exp(-\lambda(K)t-\frac{(x-1)^2}{2})\Big(v_{i}(t,x)-u_{i}(t,x)\Big)\leq 0,
\end{eqnarray*}
then for all $(t,(x,i))\in[0,T]\times\mathcal{J}^K$:
\begin{eqnarray*}
&\displaystyle\exp(-\lambda(K)t-\frac{(x-1)^2}{2})\Big(v_{i}(t,x)-u_{i}(t,x)\Big)\leq 0.
\end{eqnarray*}
So
\begin{eqnarray*}
\displaystyle\exp(-\frac{(x-1)^2}{2})\Big(v_{i}(t,x)-u_{i}(t,x)\Big)\leq 0.
\end{eqnarray*}
Finally we have, for all $(t,(x,i))\in[0,T]\times\mathcal{J}^K$:
\begin{eqnarray*}
\max\Big(0,\exp(-\frac{(x-1)^2}{2})\Big(v_{i}(t,x)-u_{i}(t,x)\Big)\Big)\leq\exp(-\frac{(K-1)^2}{2})\Big(||u||_{L^{\infty}(\mathcal{J}_T)}+||v||_{L^{\infty}(\mathcal{J}_T)}\Big).
\end{eqnarray*}
Sending $K\to \infty$ and using the boundedness of $u$ and $v$, we deduce the inequality $v\leq u$ in $[0,T]\times\mathcal{J}$. 
\end{proof}
\begin{Theorem}\label{th : exis para born infini }
Assume $(\mathcal{P}_{\infty})$. The following parabolic problem with Neumann boundary condition at the vertex:
\begin{eqnarray}\label{eq : pde parab bor infini }\begin{cases}\partial_tu_i(t,x)-\sigma_i(x,\partial_xu_i(t,x))\partial_x^2u_i(t,x)+\\
H_i(x,u_i(t,x),\partial_xu_i(t,x))~~=~~0,~~\text{ if } (t,x) \in (0,T)\times (0,+\infty),\\
F(u(t,0),\partial_x u(t,0))~~=~~0, ~~ \text{ if } t\in[0,T),\\
\text{with}~~u(t,0)=(u_1(t,0),\ldots,u_I(t,0)),~~\partial_x u(t,0)=(\partial_xu_1(t,0),\ldots,\partial_xu_I(t,0)),\\
\text{and}~~\forall (i,j)\in\{1\ldots I\}^2,~~u_i(t,0)=u_j(t,0),\\
\forall i\in\{1\ldots I\},~~ u_i(0,x)~~=~~g_i(x),~~ \text{ if } x\in [0,+\infty),\\
\end{cases}
\end{eqnarray}is uniquely solvable in the class $\mathcal{C}^{\frac{\alpha}{2},1+\alpha}(\mathcal{J}_T)\cap \mathcal{C}^{1+\frac{\alpha}{2},2+\alpha}(\overset{\circ}{\mathcal{J}_T})$.
There exist constants $(M_1,M_2,M_3)$, depending only the data introduced in assumption $(\mathcal{P}_{\infty})$ 
\begin{eqnarray*}
&M_1=M_1\Big(\max_{i\in\{1\ldots I\}}\Big\{~~\sup_{x\in(0,+\infty)}|-\sigma_i(x,\partial_xg_i(x))\partial_x^2g_i(x)+H_{i}(x,g_i(x),\partial_xg_i(x))|\Big\},\\
&\max_{i\in\{1\ldots I\}} |g_i|_{(0,+\infty)},C_H\Big),\\
&M_2=M_2\Big(\overline{\nu},\underline{\nu},\mu(M_1),\gamma(M_1),\varepsilon(M_1),\sup_{|p|\ge 0}P(M_1,|p|),\max_{i\in\{1\ldots I\}}|\partial_xg_i|_{(0,+\infty)},M_1\Big),
\\
&M_3=M_3\Big(M_1,\underline{\nu}(1+|p|)^{m-2},\mu(|u|)(1+|p|)^m,~~|u|\leq M_1,~~|p|\leq M_2\Big),
\end{eqnarray*}
such that:
\begin{eqnarray*}
&||u||_{\mathcal{C}(\mathcal{J}_T)}\leq M_1,~~||\partial_xu||_{\mathcal{C}(\mathcal{J}_T)}\leq M_2,~~||\partial_tu||_{\mathcal{C}(\mathcal{J}_T)}\leq M_1,~~||\partial_xu||_{\mathcal{C}(\mathcal{J}_T)}\leq M_3.
\end{eqnarray*}
Moreover, there exists a constant $M(\alpha)$ depending on $\Big(\alpha,M_1,M_2,M_3\Big)$ such that for any $a\in (0,+\infty)^I$:
\begin{eqnarray*}
&||u||_{\mathcal{C}^{\frac{\alpha}{2},1+\alpha}(\mathcal{J}_T^a)}\leq M(\alpha).
\end{eqnarray*}
\begin{proof}
Assume $(\mathcal{P}_{\infty})$ and let $a=(a,\ldots,a)\in(0,+\infty)^I$. Applying Theorem \ref{th : exis para}, we can define $u^a\in\mathcal{C}^{0,1}(\mathcal{J}_T^a)\cap \mathcal{C}^{1,2}(\overset{\circ}{\mathcal{J}^a_T})$ as the unique solution of: \begin{eqnarray}\label{eq : pde pour bord infini appro}\begin{cases}\partial_tu_i(t,x)-\sigma_i(x,\partial_xu_i(t,x))\partial_x^2u_i(t,x)+\\
H_i(x,u_i(t,x),\partial_xu_i(t,x))~~=~~0,~~\text{ if } (t,x) \in (0,T)\times (0,a),\\
F(u(t,0),\partial_x u(t,0))~~=~~0, ~~ \text{ if } t\in[0,T),\\
\forall i\in\{1\ldots I\},~~ u_i(t,a)~~=~~g_i(a),~~ \text{ if } t \in [0,T],\\
\forall i\in\{1\ldots I\},~~ u_i(0,x)~~=~~g_i(x),~~ \text{ if } x\in [0,a].\\
\end{cases}
\end{eqnarray}
Using assumption ($\mathcal{P}_{\infty}$) and Theorem \ref{th : exis para}, we get that there exists a constant $C>0$ independent of $a$ such that:
\begin{eqnarray*}
&\sup_{a\ge 0}~~||u^a||_{\mathcal{C}^{1,2}(\mathcal{J}_T^a)}~~\leq ~~C.
\end{eqnarray*}
We are going to send $a$ to $+\infty$ in \eqref{eq : pde pour bord infini appro}.
Following the same argument as for the proof of Theorem \ref{th : exis para}, we get that, up to a sub sequence, $u^a$ converges locally uniformly to some map $u$ which solves \eqref{eq : pde parab bor infini }. On the other hand, uniqueness of $u$ is a direct consequence of the comparison Theorem \ref{th : para comparison th jonct non bornée}, since $u\in L^{\infty}(\mathcal{J}_T)$. 
Finally the expression of the upper bounds of the derivatives of $u$ given in Theorem \ref{th : exis para born infini }, are a consequence of Theorem \ref{th : exis para} and assumption $(\mathcal{P}_{\infty})$.
\end{proof}
\end{Theorem}
\appendix 
\section{Functionnal spaces}\label{sec : functionnal spaces}
In this section, we recall several classical notations from \cite{pde para}. Let $l, T\in (0,+\infty)$ and $\Omega$ be an open and bounded subset of $\R^n$ with  smooth boundary ($n>0$). We set $\Omega_T=(0,T)\times \Omega$, and we introduce the following spaces :\\
-if $l\in 2\mathbb{N}^*$, $\Big(\mathcal{C}^{\frac{l}{2},l}(\overline{\Omega_T}),\|\cdot\|_{\mathcal{C}^{\frac{l}{2},l}(\Omega_T)}\Big)$ is the Banach space  whose elements are continuous functions $(t,x)\mapsto u(t,x)$ in $\overline{\Omega_T}$, together with all its derivatives of the form $\partial_t^r\partial_x^su$, with $2r+s<l$.
The norm $\|\cdot\|_{\mathcal{C}^{\frac{l}{2},l}(\Omega_T)}$ is defined for all $u\in \mathcal{C}^{\frac{l}{2},l}(\overline{\Omega_T})$ by:
\begin{eqnarray*}
\|u\|_{\mathcal{C}^{\frac{l}{2},l}(\Omega_T)}~~=~~\displaystyle\sum_{2r+s=j}\sup_{(t,x)\in \Omega_T}|\partial_t^r\partial^s_xu(t,x)|.
\end{eqnarray*}
-if $l\notin \mathbb{N}^*$, $\Big(\mathcal{C}^{\frac{l}{2},l}(\overline{\Omega_T}),\|.\|_{\mathcal{C}^{\frac{l}{2},l}(\Omega_T)}\Big)$ is the Banach space  whose elements are continuous functions $(t,x)\mapsto u(t,x)$ in $\overline{\Omega_T}$, together with all its derivatives of the form $\partial_t^r\partial_x^su$, with $2r+s<l$, and satisfying an H\"{o}lder condition with exponent $\frac{l-2r-s}{2}$ in their first variable, and with exponent $(l-\lfloor l \rfloor)$ in their second variable, over all the connected components of $\Omega_T$ whose radius is smaller than $1$.

The norm $\|\cdot\|_{\mathcal{C}^{\frac{l}{2},l}(\Omega_T)}$ is defined for all $u\in \mathcal{C}^{\frac{l}{2},l}(\overline{\Omega_T})$ by:
\begin{eqnarray*}
\|u\|_{\mathcal{C}^{\frac{l}{2},l}(\Omega_T)}~~=~~|u|_{\Omega_T}^{l}~~+~~\sum_{j=0}^{\lfloor l \rfloor} |u|_{\Omega_T}^{j},
\end{eqnarray*}
with:
\begin{eqnarray*}
&\forall j\in\{0,\ldots,l\},~~|u|_{\Omega_T}^{j}~~=~~\displaystyle\sum_{2r+s=j}\sup_{(t,x)\in \Omega_T}|\partial_t^r\partial^s_xu(t,x)|,\\
&|u|_{\Omega_T}^{l}~~=~~|u|_{x,\Omega_T}^{l}~~+~~|u|_{t,\Omega_T}^{\frac{l}{2}},\\ 
&|u|_{x,\Omega_T}^{l}~~=~~\displaystyle\sum_{2r+s=\lfloor l \rfloor}|\partial_t^r\partial^s_xu(t,x)|_{x,\Omega_T}^{l-\lfloor l \rfloor},\\
&|u|_{t,\Omega_T}^{l}~~=~~\displaystyle\sum_{0< l-2r-s<2}|\partial_t^r\partial^s_xu(t,x)|_{t,\Omega_T}^{\frac{l-2r-s}{2}},\\
&|u|^{\alpha}_{x,\Omega_T}~~ =~~\underset{t\in (0,T)}{\sup}~~\underset{x,y\in \Omega, x\neq y, |x-y|\leq 1}{\sup}\displaystyle\frac{|u(t,x)-u(t,y)|}{|x-y|^{\alpha}},~~0<\alpha<1,
\\
&|u|^{\alpha}_{t,\Omega_T}~~ =~~ \underset{x\in \Omega}{\sup}~~\underset{t,s\in (0,T), t\neq s, |t-s|\leq 1}{\sup}\displaystyle\frac{|u(t,x)-u(s,x)|}{|t-s|^{\alpha}},~~0<\alpha<1.
\end{eqnarray*}
- $\mathcal{C}^{\frac{l}{2},l}(\Omega_T)$ is the set whose elements $f$ belong to $\mathcal{C}^{\frac{l}{2},l}(\overline{O_T})$ for any open set $O_T$ separated from the boundary of $\Omega_T$ by a strictly positive distance, namely: 
\begin{eqnarray*}
\underset{y\in\partial \Omega_T, x \in \overline{O_T}}{\inf}~~||x-y||_{\R^n}~~>~~0.
\end{eqnarray*}
- $\mathcal{C}_b^{\frac{l}{2},l}(\Omega_T)$ is the subset of $\mathcal{C}^{\frac{l}{2},l}(\Omega_T)$ consisting in maps $u$ such that the derivatives of the form $\partial_t^r\partial_x^su$, (with $2r+s<l$) are bounded, namely $\sup_{(t,x)\in\Omega_T}|\partial_t^r\partial_x^su(t,x)| <+\infty$. 

We use the same notations when the domain does not depend on time, namely $T=0$, $\Omega_T=\Omega$, just removing the dependence on the time variable. \\
For $R>0$, we denote by $L^2((0,T)\times (0,R))$ the usual space of square integrable maps and by $\mathcal{C}_c^\infty((0,T)\times (0,R))$ the set of infinite continuous differentiable functions on $(0,T)\times (0,R)$, with compact support.  
\section{The Elliptic problem}\label{sec : appendix elliptic problem}
\begin{Proposition}\label{pr : conti para}
Let $\theta\in\R$, $i\in \{1, \dots, I\}$ and assume  ($\mathcal{E}$) holds. Let $u^\theta_i\in \mathcal{C}^{2}([0,a_i])$ be the solution of:
\begin{eqnarray}\label{eq : ell Diri pblm}
\begin{cases}-\sigma_i(x,\partial_x u^\theta_i(x))\partial_x^2u^\theta_i(x)+H_i(x,u^\theta_i(x),\partial_xu^\theta_i(x))~~=~~0, \text{ if } x \in (0,a_i)\\
u^\theta_i(0)~~=~~u^\theta(0)~~=~~\theta,
\\u^\theta_i(a_i)~~=~~\phi_i.
\end{cases}
\end{eqnarray}
Then the following map: 
\begin{eqnarray*}
\Psi :=\begin{cases}
\R \to \mathcal{C}^{2}([0,a_i])\\
\theta \mapsto u^\theta_i
\end{cases},
\end{eqnarray*}
is continuous.
\end{Proposition}
\begin{proof}
Let $\theta_n$ a sequence converging to $\theta$. Using the Schauder estimates Theorem 6.6 of \cite{Gilbarg}, we get that there exists a constant $M>0$ independent of $n$, depending only the data,
such that for all $\alpha \in (0,1)$:
\begin{eqnarray*}
\|u^{\theta_n}_i\|_{\mathcal{C}^{2+\alpha}([0,a_i])}&\leq& M.
\end{eqnarray*}
From Ascoli's Theorem, $u^{\theta_n}_i$ converges up to a subsequence to $v$ in $\mathcal{C}^{2}([0,a_i])$ solution of \eqref{eq : ell Diri pblm}. By uniqueness of the solution of \eqref{eq : ell Diri pblm}, $u^{\theta_n}_i$ converges necessary to the solution $u^{\theta}_i$ of \eqref{eq : ell Diri pblm} in $\mathcal{C}^{2}([0,a_i])$, which completes the proof. 
\end{proof}
\textbf{Proof of Theorem \ref{th:exi elli}.}
\begin{proof}
The uniqueness of \eqref{eq base 2} results from the elliptic comparison Theorem \ref{th: Ellip Th Comp}.

We turn to the solvalbility, and for this let $\theta \in \R$. We consider the elliptic Dirichlet problem posed on the the junction:
\begin{eqnarray}\label{eq : l1}
\begin{cases}-\sigma_i(x,\partial_x u_i(x))\partial_x^2u_i(x)+H_i(x,u_i(x),\partial_xu_i(x))~~=~~0, ~~ \text{ if } x \in (0,a_i),\\
\forall i\in\{1\ldots I\},~~u_i(0)~~=~~u(0)~~=~~\theta,
\\u_i(a_i)~~=~~\phi_i.
\end{cases}
\end{eqnarray}
For all $i\in\{1\ldots I\}$, each elliptic problem is uniquely solvable on each edge in $\mathcal{C}^{2+\alpha}([0,a_i])$, then \eqref{eq : l1} is uniquely solvable in the class $\mathcal{C}^{2+\alpha}(\mathcal{J}^a)$, and we denote by $u^\theta$ its solution.

We turn to the Neumann boundary condition at the vertex. Let us recall assumption $(\mathcal{E})$(i)
$$\begin{cases}
F \text{ is decreasing in its first variable, nondecreasing in its second variable},\\
\text{or }F \text{ is nonincreasing in its first variable, increasing in its second variable},
\\
\exists (b,B)\in\R^I\times \R^I,  \text{ such that : }F(b,B)~~=~~0.\\
\end{cases}$$
Fix now:
\begin{eqnarray*}
K_i~~=~~\sup_{(x,u)\in(0,a_i)\times(-a_iB_i,a_iB_i) }|H_i(x,u,B_i)|,\\
\theta~~\ge \underset{i \in\{1\ldots I\}}{\max}\Big\{|b_i|+|\phi_i|+|a_iB_i|+ \frac{K_i}{C_H}\Big\},
\end{eqnarray*}
and let us show that $f: x\mapsto \theta+B_ix$, is a super solution on each edge $J_{i}^{a_i}$ of \eqref{eq : l1}.

We have the boundary conditions
\begin{eqnarray*}
f(0)~~= ~~\theta, ~~f(a_i)= \theta + a_iB_i~~\ge~~ |\phi_i|+ |a_iB_i| + a_iB_i ~~\ge ~~\phi_i,
\end{eqnarray*}
and using assumption $(\mathcal{E})$ (iii), we have for all $x\in (0,a_i)$
\begin{eqnarray*}
&-\sigma_i(x,\partial_x f(x))\partial_x^2f(x)+H_i(x,f(x),\partial_xf(x))~~=~~H_i(x,\theta+B_ix,B_i)~~\ge~~~H_i(x,B_ix,B_i)~~\\
&+~~C_H\theta~~\ge~~H_i(x,B_ix,B_i)+K_i~~\ge~~0.
\end{eqnarray*} 
We then get that for each $i\in\{1\ldots I\}$, $x\in[0,a_i]$, $u_i^\theta(x)\leq \theta+B_ix$. A Taylor expansion in the neighborhood of the junction point gives that
for each $i \in\{1\ldots I\}$, $~~\partial_{x}u_i^\theta(0)\leq B_i$.\\ 
Since $u^\theta(0)\ge b$, we then get from assumption $(\mathcal{E})$ (i):
\begin{eqnarray*}
F(u^\theta(0),\partial_{x}u^\theta(0))~~\leq ~~F(b,B)~~\leq~~0.
\end{eqnarray*}
Similarly, fixing: 
\begin{eqnarray*}
\theta~~\leq~~\underset{i \in\{1\ldots I\}}{\min}\Big\{-|b_i|-|\phi_i|- |a_iB_i|- \frac{K_i}{C_H}\Big\},
\end{eqnarray*}
the map $f:x\mapsto \theta+xB_i$ is a sub solution on each vertex $J_{i}^{a_i}$ of \eqref{eq : l1}, then for each $i \in\{1\ldots I\}$, $\partial_{x}u_i^\theta(0)\ge B_i$, which means: 
\begin{eqnarray*}
F(u^\theta(0),\partial_{x}u^\theta(0))~\ge~~ 0.
\end{eqnarray*}
From Proposition \ref{pr : conti para}, we know that the real maps $\theta\mapsto u^\theta(0)$ and $\theta\mapsto\partial_{x}u^\theta(0)$ are  continuous. Using the continuity of $F$ (assumption $(\mathcal{E})$), we get that $\theta \mapsto F(u^\theta(0),\partial_xu^\theta(0))$ is continuous, and therefore there exists $\theta^{*}\in \R$ such that: \begin{eqnarray*}
F(u^{\theta^*}(0),\partial_{x}u^{\theta^*}(0))~~=~~0.
\end{eqnarray*}
We remark that $\theta^*$ is bounded by the data, namely $\theta^*$ belongs to the following interval:
\begin{eqnarray*}
&\Big[~~\underset{i \in\{1\ldots I\}}{\min}\Big\{-|b_i|-|\phi_i|-|a_iB_i|- \frac{\sup_{(x,u)\in(0,a_i) }|H_i(x,B_ix,B_i)|}{C_H}\Big\},
\\&\underset{i \in\{1\ldots I\}}{\max}\Big\{ |b_i|+ |\phi_i|+|a_iB_i|+ \frac{\sup_{(x,u)\in(0,a_i) }|H_i(x,B_ix,B_i)|}{C_H}\Big\}~~\Big].\\
\end{eqnarray*} 
This completes the proof.
Finally, since the solution $u^{\theta^*}$ of \eqref{eq base 2} is unique, we get the uniqueness of $\theta^*$.
\end{proof}

\end{document}